\documentclass[12pt]{article}
\usepackage{amsfonts}
\usepackage{amssymb}
\usepackage{mathrsfs}
\usepackage{srcltx}
\usepackage[monochrome]{xcolor}
\usepackage{amsmath}
\usepackage{amsthm}
\usepackage{amstext}
\usepackage{amsopn}
\usepackage{placeins}
\usepackage{enumitem}

\usepackage{appendix}
\usepackage{subfigure}
\usepackage{tikz}

\usepackage{graphicx}
\usepackage{bm}
\newtheorem{theorem}{Theorem}[section]
\newtheorem{lemma}[theorem]{Lemma}
\newtheorem{proposition}[theorem]{Proposition}
\newtheorem{corollary}[theorem]{Corollary}

\theoremstyle{definition}

\theoremstyle{remark}
\newtheorem{remark}[theorem]{Remark}

\numberwithin{equation}{section}

\newcommand{\ba}{\begin{array}}
\newcommand{\ea}{\end{array}}
\newcommand{\f}{\frac}

\newcommand{\la}{\lambda}
\newcommand{\R}{{\mathbb R}}
\newcommand{\ds}{\displaystyle}

\begin{document}
\date{}
\title{ \bf\large{Evolution of dispersal in advective patchy environments}\thanks{S. Chen is supported by National Natural Science Foundation of China (Nos. 12171117, 11771109) and Shandong Provincial Natural Science Foundation of China (No. ZR2020YQ01),  J. Shi is supported by US-NSF grant DMS-1715651 and DMS-1853598, and Z. Shuai is supported by US-NSF grant DMS-1716445. 
}}
\author{
Shanshan Chen\footnote{Email: chenss@hit.edu.cn}\\[-1mm]
{\small Department of Mathematics, Harbin Institute of Technology}\\[-2mm]
{\small Weihai, Shandong 264209, P. R. China}\\[2mm]
Junping Shi\footnote{Corresponding author.  Email: jxshix@wm.edu}\\[-1mm]
{\small Department of Mathematics, William \& Mary}\\[-2mm]
{\small Williamsburg, Virginia 23187-8795, USA}\\[2mm]
Zhisheng Shuai\footnote{Email: shuai@ucf.edu}\\[-1mm]
{\small Department of Mathematics, University of Central Florida}\\[-2mm]
{\small Orlando, Florida 32816, USA}\\[2mm]
Yixiang Wu\footnote{Email: yixiang.wu@mtsu.edu} \\[-1mm]
{\small Department of Mathematics, Middle Tennessee State University}\\[-2mm]
{\small Murfreesboro, Tennessee 37132, USA}
}

\maketitle


\begin{abstract}
We study a two-species competition model in a patchy advective environment, where the species are subject to both directional drift and undirectional random dispersal between patches and there are losses of individuals in the downstream end (e.g., due to the flow into a lake or ocean). The two competing species are assumed to have the same growth rates but different advection and random dispersal rates. We focus our studies on the properties of an associated eigenvalue problem which characterizes the extinction/persistence dynamics of the underlying patch population model. We also derive conditions on the advection and random dispersal rates under which a mutating species can or cannot invade the resident species.

\end{abstract}

\section{Introduction}
The organisms in streams are subject to both directional drift and undirectional random diffusion. Intuitively, the stream flow takes the organisms to the downstream locations which are often fatal to them, while random diffusion may drive them to favorable locations in the upstream. How the joint force of directed and undirectional movements affects the extinction and persistence of a biological species have attracted the attention of many researchers \cite{HuangJin, jin2011seasonal, lou2014evolution,  lutscher2006effects, lutscher2007spatial, lutscher2005effect, speirs2001population}.

In the framework of discrete patch models, a population in a stream environment with logistic type growth can be described by the following system (\cite{chen2019spectral,cosner1996variability,li2010global,Lu1993}):
\begin{equation}\label{pat-s}
\begin{cases}
\ds\frac{du_i}{dt}=\sum_{j=1}^nL_{ij}u_j+u_i(r_i-u_i),&i=1,\dots,n,\;\;t>0,\\
\bm u(0)=\bm u_0\ge(\not\equiv)\bm \ 0,
\end{cases}
\end{equation}
where $n\ge 2$ is the number of patches; $\bm u=(u_1,\dots,u_n)$, and $u_i$ denotes the population density in patch $i$; $L_{ij}$ is the movement rate of individuals from patch $j$ to patch $i$; and $r_i$ is the  intrinsic growth rate in patch $i$. The connection matrix $L=\left(L_{ij}\right)$ depends on the topology of the stream and the directional and undirectional movement rates of the species. In this paper, we only consider the population dynamics in a stream with free flow from upstream end (patch $i=1$) to the downstream end (patch $i=n$). The following three ecological scenarios at the downstream end are typical \cite{lou2014evolution,lutscher2006effects, speirs2001population}:
\begin{enumerate}
 \item [(i)] \textbf{Stream to lake}. The lake environment is as favorable as the stream environment for the species, and individuals can return to the stream from the lake by diffusion. Moreover, the diffusive flux into and from the lake balances;
   \item [(ii)] \textbf{Stream to ocean}. The ocean environment is fatal to the species in the stream, and individuals cannot return to the stream  from the ocean;
 \item [(iii)] \textbf{Inland stream}. Individuals cannot move in
or out through the downstream end.
 \end{enumerate}
The above cases (i)-(iii) correspond to three types of movements  at the downstream end, see $(a)$-$(c)$ in Fig. \ref{fig1}.
\begin{figure}[htbp]
\centering\includegraphics[width=0.6\textwidth]{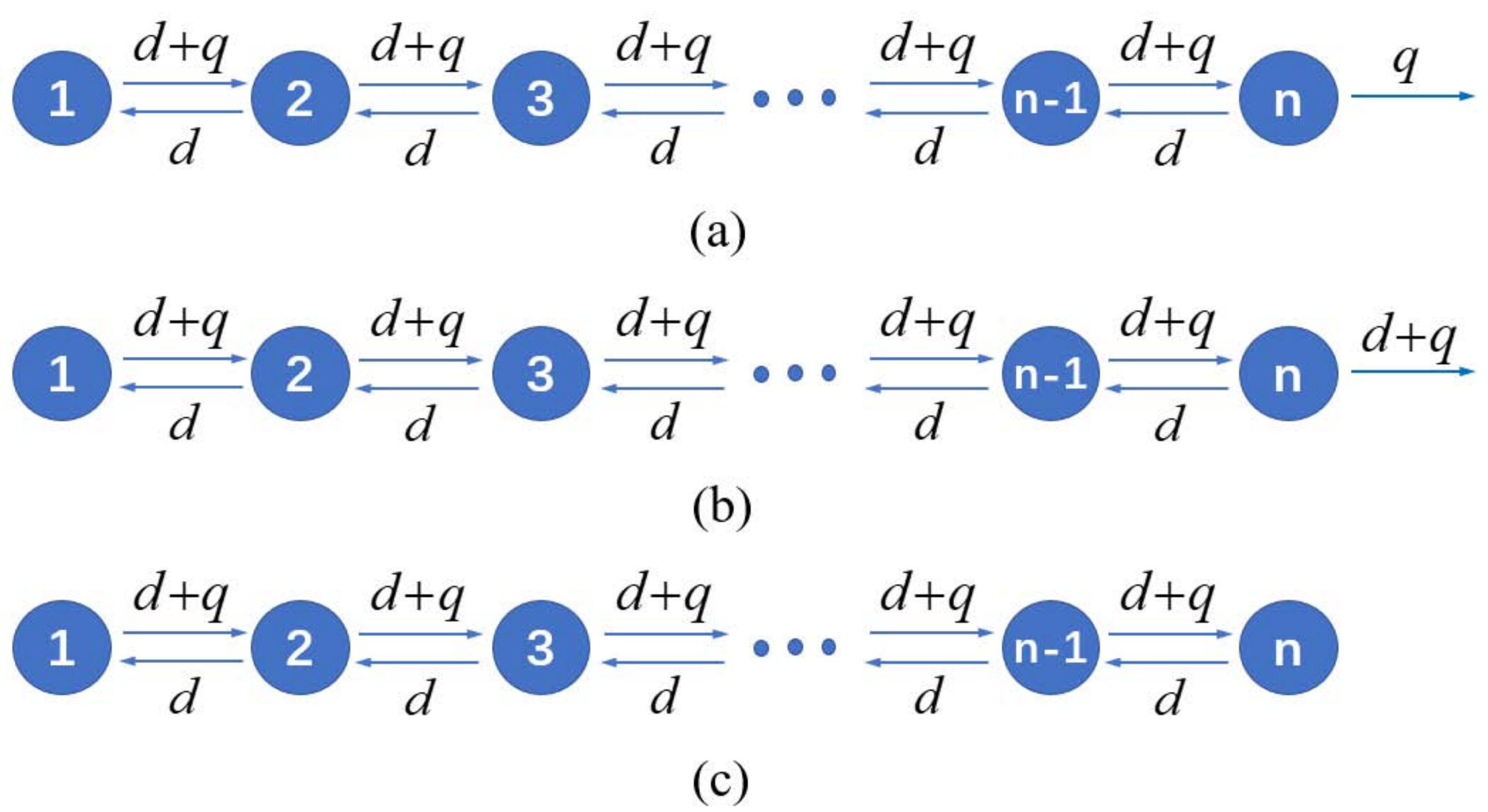}
\caption{Advective and diffusive movement of the species among patches. Here, $(a)$: stream to lake; $(b)$: stream to ocean; and $(c)$: inland stream. Here $d>0$ is the diffusion rate, and $q\geq 0$ is the advection rate.
  \label{fig1}}
\end{figure}

The movement of the species among patches in Fig. \ref{fig1} can be described by an $n\times n$ matrix $L=d D+q Q$ with $d$ and $q$ being the diffusion and advection rates, respectively,  and $D=(D_{ij})$ representing the diffusion pattern and $Q=(q_{ij})$ describing the directed movement pattern of individuals.  Then the matrices $D$ and $Q$ satisfy one of the following three assumptions:
\begin{itemize}[leftmargin=*,align=left]
\item[\textbf{H}1.]  Case $(a)$: stream to lake. The matrix $D=(D_{ij})$ is given by
\begin{equation}\label{M}
\begin{split}
D_{ij}=&\begin{cases}
1,&i=j-1\;\text{or}\;i=j+1,\\
-2,&i=j=2,\dots,n-1,\\
-1,&i=j=1,n,\\
0, &\text{otherwise},
\end{cases}
\end{split}
\end{equation}
and  $Q=(Q_{ij})$ is given by
\begin{equation}\label{A}
\begin{split}
Q_{ij}=&\begin{cases}
1,&i=j+1,\\
-1,&i=j=1,\dots,n,\\
0,&\text{otherwise}.
\end{cases}
\end{split}
\end{equation}

\item[\textbf{H}1$^*$.]  Case (b): stream to ocean. The same as \textbf{H}1 except that $D_{nn}=-2$.

\item[\textbf{H}1$^{**}$.] Case (c): inland stream. The same as \textbf{H}1 except that  $Q_{nn}=0$.
\end{itemize}

We will consider the following two-species Lotka-Volterra competition model in a spatially homogeneous patchy stream environment:
\begin{equation}\label{pat-cp}
\begin{cases}
\ds\frac{du_i}{dt}=\ds\sum_{j=1}^{n} (d_1 D_{ij}+q_1Q_{ij})u_j+u_i(r-u_i-v_i), &i=1,\dots,n,\;\;t>0,\\
\ds\frac{dv_i}{dt}=\ds\sum_{j=1}^{n}(d_2 D_{ij}+q_2Q_{ij})v_j+v_i(r-u_i-v_i),&i=1,\dots,n,\;\; t>0, \\
\bm u(0)=\bm u_0\ge(\not\equiv)\;\bm 0,\;\bm v(0)=\bm v_0\ge(\not\equiv)\;\bm0.
\end{cases}
\end{equation}
Here the growth rate $r$ is assumed to be a positive constant for all patches, and the two species have the same interspecific and intraspecific competition coefficients (normalized to be $1$ for simplicity). So the two competing species are identical except their diffusion and advection rates. The  connection matrices are
\begin{equation}\label{Lk}
L^{(k)}=\left(L^{(k)}_{ij}\right)=d_kD+q_kQ, \;\; k=1,2,
\end{equation}
where $\left(D_{ij}\right)$ and $\left(Q_{ij}\right)$ are defined in \eqref{M} and \eqref{A} for case (a)-(c), respectively. For the purpose of brevity, we will restrict our attention to cases (a) and (b) where there are losses of individuals in the downstream end, and case (c) will be studied in a follow-up paper.

Our work is largely motivated by previous researches on population dynamics in streams in the framework of partial differential equation models \cite{cantrell2004spatial, lam2015evolution, lam2016emergence,lou2014evolution, lou2016qualitative, lou2015evolution, ma2020evolution, vasilyeva2012flow, zhao2016lotka, zhou2016lotka}. The reaction-diffusion-advection model analogous to \eqref{pat-s} is
\begin{equation}\label{con-s}
\begin{cases}
u_t=\tilde{d}u_{xx}-\tilde{q} u_x+u[r(x)-u], &0<x<l,\;\;t>0,\\
\tilde{d} u_x(0,t)-\tilde{q} u(0,t)=0,& t>0, \\
\tilde{d} u_x(l,t)-\tilde{q} u(l,t)=-\beta \tilde{q} u(l,t),& t>0,\\
u(x,0)\ge(\not \equiv)0.
\end{cases}
\end{equation}
Here the species $u$ lives in a stream represented by $0\le x\le l$; $\tilde{d}$ is the diffusion rate and $\tilde{q}$ is the advection rate of the species;  $x=0$ is the upstream end and $x=l$ is the downstream end. The no-flux boundary condition is imposed at the upstream end $x=0$, which means that individuals cannot move in or out through the the upstream boundary. A parameter $\beta$ is introduced for the boundary condition at the downstream end $x=l$ to measure the loss rate of individuals. The corresponding three cases in Fig. \ref{fig1} are as follows: $(1)$ free-flow boundary condition $u_x(l,t)=0$ for $\beta=1$; $(2)$ hostile boundary condition $u(l,t)=0$ for $\beta\to\infty$; and $(3)$ no-flux boundary condition $\tilde{d} u_x(l,t)-\tilde{q}  u(l,t)=0$ for $\beta=0$. 

The reaction-diffusion-advection version of the two species competition model \eqref{pat-cp} over a stream with different boundary conditions in the downstream end has been studied by many authors \cite{lam2015evolution,lou2014evolution, lou2016qualitative, lou2015evolution, ma2020evolution, vasilyeva2012flow, zhao2016lotka, zhou2016lotka}.
In the seminal work of Hastings \cite{hastings1983can} and Dockery \emph{et al.} \cite{dockery1998evolution}, it has been shown that a mutating species can invade if and only if it has a smaller diffusion rate when both species have no directional movement and are identical except for the diffusion rates. However in a stream environment with free-flow boundary conditions \cite{lou2014evolution,vasilyeva2012flow,zhou2018global} or inland boundary conditions \cite{lam2015evolution,lou2018coexistence,lou2016qualitative, lou2015evolution,zhou2016lotka}, the species with larger diffusion rate and/or smaller advection rate may be selected. The Dirichlet boundary condition case seems to be less studied, and the authors in \cite{yan2022competition} showed that both coexistence and bi-stability are possible numerically.

Model \eqref{pat-s} is a discrete version of  \eqref{con-s}. Indeed, if we divide the interval $[0, l]$ into $n$ sub-intervals with equal length $\Delta x=l/n$ and endpoints $0, 1,..., n$.  At endpoints $i=1, ..., n$, we discretize $u_{xx}$ and $u_x$ to obtain the following equation:
\begin{equation}\label{dis}
    \frac{ d u_i}{dt}=\tilde{d}\frac{u_{i+1}-2u_i+u_{i-1}}{(\Delta x)^2}-\tilde{q}\frac{u_i-u_{i-1}}{\Delta x}+u_i(r_i-u_i), \ \ i=1,..., n,
\end{equation}
where $u_i(t)$ is the population density at endpoint $i$. Note that for $i=2,\cdots,n-1$,  \eqref{dis} is the same as \eqref{pat-s} with $d=\tilde{d}/(\Delta x)^2$ and $q=\tilde{q}/\Delta x$. At the upstream end $x=0$, we discretize the no-flux boundary condition to obtain
$$
\tilde{d}\frac{u_1-u_0}{\Delta x}-\tilde{q}u_0=0.
$$
Substituting it into \eqref{dis} for $i=1$, we obtain \eqref{pat-s} for $i=1$. If the downstream end $x=l$ is imposed with the free-flow boundary condition, then the discrete version is
$$
\tilde{d}\frac{u_{n+1}-u_n}{\Delta x}=0.
$$
Substituting it into \eqref{dis} for $i=n$, we obtain \eqref{pat-s} for $i=n$.  If the downstream end is associated with the hostile boundary condition, we view an extra point $n+1$ as the endpoint of the stream and obtain $u_{n+1}=0$. Again, we can substitute it into \eqref{dis} for $i=n$ to obtain the corresponding equation in  \eqref{pat-s}. The no-flux boundary condition at $x=l$ can be treated similarly.

The discrete patch model \eqref{pat-s} and the two-species competition model \eqref{pat-cp} with the dispersal matrix $D,Q$ defined as in \eqref{M}-\eqref{A} approximate the reaction-diffusion-advection model \eqref{con-s} and the corresponding two-species competition model as the number of patches gets large and the total length of the system remains fixed. Similar approach for spatial population dynamics has been used in \cite{deangelis2016dispersal,keitt2001allee,levin1976population, Owen2001}, and comparison of continuous and discrete space models were also made in these work. While the two types of models often produce similar results, it is known that they can also have different outcomes when there is Allee effect in the system \cite{keitt2001allee,Owen2001}.

We will investigate model \eqref{pat-cp} in the approach of adaptive dynamics framework \cite{dieckmann1996dynamical, geritz1998evolutionarily}, which is the method adopted in \cite{lou2014evolution,vasilyeva2012flow}. For this purpose, we will first impose conditions on $d_1$ and $q_1$ such that  $\bm u-$species is established as a semi-trivial equilibrium $E_1=(\bm u^*, \bm 0)$ (the $\bm v-$only equilibrium will be denoted by $E_2=(\bm 0, \bm v^*)$) when there is no $\bm v-$species in the system. Then we investigate the stability/instability of $E_1$ when $d_2$ and $q_2$ varies. We show that there is a curve $q=q^*_{\bm{r-u^*}}(d)$ passing through $(d_1, q_1)$ and dividing the $d-q$ plane into two regions such that $E_1$ is stable if $(d_2, q_2)$ is above the curve while $E_1$ is unstable if $(d_2, q_2)$ is below it. When the downstream end is coupled with no-flux boundary conditions (case (a)), we show that the curve $q=q^*_{\bm{r-u^*}}(d)$ is strictly increasing, and $\bm v$ species can invade if it has larger diffusion or smaller advection rate. If the hostile boundary conditions (case (b)) are imposed at the downstream end, we show that smaller advection rate is selected. If $q_2=q_1$ and $d_2$ is close to $d_1$,   $[q^*_{\bm{r-u^*}}(d)]'|_{d=d_1}>(<)0$ means that species $\bm v$ can invade if and only if $d_2> (<) d_1$. We show that $\left[q_{\bm{r-u^*}}^*(d)\right]'|_{d=d_1}$ changes sign in case (b) as $(d_1, q_1)$ varies, which indicates whether smaller or larger diffusion rate is favored depends (for some advection rate, there seems to be an intermediate diffusion rate which is a convergence stable strategy). For both cases, we find parameter ranges of $d_2$ and $q_2$ such that competitive exclusion happens. We also give conditions under which coexistence or bi-stability of the two species occurs.

The global dynamics of the single species patch model \eqref{pat-s} is well-known. In \cite{cosner1996variability,li2010global,Lu1993}, it has been shown that  either the trivial equilibrium of \eqref{pat-s} is globally stable or the model  has a globally asymptotically stable positive equilibrium. There are also many research works on the two-species competition model \eqref{pat-cp}, especially when the number of patches $n$ is small ($n=2$ or $3$). We refer interested readers to the works on two-patch models without directional dispersal \cite{cheng2019coexistence,gourley2005two,lin2014global}  and the ones with directional dispersal \cite{Hamida2017,jiang2020two,Jiang-Lam-Lou2021,lou2019ideal, noble2015evolution,xiang2019evolutionarily}. More recently, the competition on a river network was considered for three-patch models \cite{jiang2020two,Jiang-Lam-Lou2021}. When $(d_1, q_1)$ is a multiple of $(d_2, q_2)$, complete global dynamics of \eqref{pat-cp} has been classified in our recent work \cite{chen2021global}. We refer to \cite{stephen2007ideal, cantrell2012evolutionary, cantrell2017evolution, kirkland2006evolution, levin1984dispersal,  mcpeek1992evolution, smith2008monotone} and the references therein for more works on competition models in patchy environment. 

Our paper is organized as follows. In Section 2, we introduce the terminology and state some useful results; in Section 3, we study the properties of the principal eigenvalue of an associated eigenvalue problem which determines the existence/nonexistence  of positive equilibrium of \eqref{pat-s}; and in Section 4, we perform invasion analysis for the two-species competition model \eqref{pat-cp}.  In Section 5, we present some numerical simulations and formulate some conjectures on coexistence and bi-stability about the solutions of the model.

\section{Preliminaries}
Let $\bm u=(u_1,\dots,u_n)^T\in \R^n$ be a vector. We write $\bm u\gg\bm0$ ($u\ge\bm0$) if $u_i>0$ ($u_i\ge 0$) for all $1\le i\le n$, and $\bm u>\bm0$ if $\bm u\ge \bm0$ and $\bm u\neq \bm0$. Let $A=(a_{ij})_{n\times n}$ be a real-valued square matrix and let $\sigma(A)$ be the set of all eigenvalues of $A$.
The {\it spectral bound} $s(A)$ of $A$ is defined as
$$
s(A)=\max\{{\rm Re} (\lambda): \lambda\in\sigma(A)\}.
$$
The matrix $A$ is {\it reducible} if we can partition $\{1, 2, \dots, n\}$ into two nonempty subsets $E$ and $F$ such that $a_{ij}=0$ for all $i\in E$ and $j\in F$. Otherwise $A$ is {\it irreducible}. A real-valued square matrix $A$ is called \emph{essentially nonnegative} if all its off-diagonal entries are nonnegative. If $A$ is an irreducible essentially nonnegative matrix, then by the Perron-Frobenius Theorem \cite{Li2002JMB}, $s(A)$ is an eigenvalue of $A$ (called the principal eigenvalue), which is the unique eigenvalue corresponding with a positive eigenvector.  It is easy to see that $D$ and $L$ defined in the Introduction are irreducible and essentially nonnegative. If $D$ satisfies \textbf{H}1 or \textbf{H}1$^{**}$, then $s(D)=0$ corresponding with a positive eigenvector $(1/n, \dots, 1/n)$; and if $D$ satisfies \textbf{H}1$^{*}$, then $s(D)<0$.

Denote by $\la_1(d,q,\bm r)$ the principal eigenvalue of the matrix $dD+qQ+diag(r_i)$, where $\bm r=(r_1,\dots,r_n)$ is a real-valued vector, so  $\la_1(d,q,\bm r)$  satisfies the following eigenvalue problem:
\begin{equation}\label{eigen}
\ds\sum_{j=1}^{n}(dD_{ij}+qQ_{ij})\phi_j+r_i\phi_i=\la\phi_i. \;\;i=1,\dots,n,\\
\end{equation}

The global dynamics of the single species model \eqref{pat-s} is determined by the sign of   $\la_1(d,q,\bm r)$ (see \cite{cosner1996variability,li2010global,Lu1993} for the proof):

\begin{lemma}\label{DS-single}
Suppose that  \textbf{H}1, \textbf{H}1$^{*}$ or \textbf{H}1$^{**}$ holds. Let $\la_1(d,q,\bm r)$ be the principal eigenvalue of \eqref{eigen}. If $\la_1(d,q,\bm r)\le0$, then the trivial equilibrium $\mathbf{0}$ of \eqref{pat-s} is globally asymptotically stable; and if $\la_1(d,q,\bm r)>0$, then model \eqref{pat-s} admits a unique positive equilibrium $\bm u^*\gg\bm 0$, which is globally asymptotically stable.
\end{lemma}
For further applications, we need the following result about the monotonicity of the spectral bound/principal eigenvalue  \cite{altenberg2012resolvent, chen2019spectral}.
\begin{lemma}\label{theorem_quasi}
Let $A=(a_{ij})_{n\times n}$ be an irreducible essentially nonnegative  matrix and  $R=\text{diag}(r_i)$ be a real diagonal  matrix.   Then the following results hold:
\begin{enumerate}
\item [{\rm (i)}] If $s(A)<0$, then
$$
\ds\frac{d}{d\mu}s(\mu A+R)<0
$$
for $\mu\in (0, \infty)$; Moreover,
 $$
 \lim_{\mu\rightarrow 0}s(\mu A+R)=\max_{1\le i\le n}\{r_i\} \text{  and }\; \lim_{\mu\rightarrow\infty}s(\mu A+R)=-\infty;
 $$
\item [{\rm (ii)}] If $s(A)=0$, then
$$
\ds\frac{d}{d\mu} s(\mu A+R)\le 0
$$
for $\mu\in (0, \infty)$ and the equality holds if and only $r_1=\dots=r_n$; Moreover,
$$
\lim_{\mu\rightarrow 0}s(\mu A+R)=\max_{1\le i\le n}r_i \;\;
\text{and} \;\; \lim_{\mu\rightarrow\infty}s(\mu A+R)=\sum_{i=1}^n{\theta_ir_i},
$$
where $\theta_i\in (0, 1)$, $1\le i\le n$, is determined by $A$ and $\ds\sum_{i=1}^n\theta_1=1$ (if $A$ has each column sum equaling zero, then $\bm\theta=(\theta_1,\dots,\theta_n)^T$ is a positive eigenvector of $A$ corresponding to eigenvalue $0$).
\end{enumerate}
\end{lemma}

We will use the monotone dynamical system theory \cite{hess,hsu1996competitive,lam2016remark,smith2008monotone} to investigate the global dynamics of the Lotka-Volterra competition system \eqref{pat-cp}. Let $X=\mathbb{R}_+^n\times\mathbb{R}_+^n$ equipped with an order $\le_K$ generated by the cone $K=\mathbb{R}_+^n\times(-\mathbb{R}_+^n)$. That is, for $\bm x=(\bm u_1, \bm v_1), \bm y=(\bm u_2, \bm v_2)\in X$, we say $\bm x\le_K \bm y$ if $\bm u_1\le \bm u_2$ and $\bm v_1\ge \bm v_2$; $\bm x<_K \bm y$ if $\bm x\le_K \bm y$ and $\bm x\neq \bm y$. 
The solutions of \eqref{pat-cp} induce a strictly monotone dynamical system in $X$: for two initial data $(\bm u_{1,0}, \bm v_{1,0})<_K(\bm u_{2,0}, \bm v_{2,0})$, the corresponding solutions of \eqref{pat-cp} satisfy $(\bm u_1(t), \bm v_1(t))<_K (\bm u_2(t), \bm v_2(t))$ for all $t\ge 0$. By the strictly monotone dynamical system theory, the global dynamics of \eqref{pat-cp} is largely determined by the local/linearized stability of the semi-trivial equilibria $E_1$ and $E_2$:

\begin{enumerate}
 \item if $E_2$ is unstable and \eqref{pat-cp} has no positive equilibrium, then $E_1$ is globally asymptotically stable;  if $E_1$ is unstable and \eqref{pat-cp} has no positive equilibrium, then $E_2$ is globally asymptotically stable;
 \item if $E_1$ and $E_2$ are both unstable, then \eqref{pat-cp} has at least one stable positive equilibrium, which is globally asymptotically stable if it is unique;
 \item if $E_1$ and $E_2$ are both locally asymptotically stable, then \eqref{pat-cp} has at least one unstable positive equilibrium.
\end{enumerate}

\section{Persistence of a single species}
In this section, we consider the mutual effects of the  diffusion and advection rates on the dynamics of the single species model \eqref{pat-s}. By Lemma \ref{DS-single}, the global dynamics of the model is determined by the sign of $\lambda_1(d, q, \bm r)$. In this section, we study the properties of $\lambda_1(d, q, \bm r)$ with respect to $d$ and $q$ in cases (a)-(b).

\subsection{Monotonicity of $\la_1(d,q,\bm r)$ in $q$}
In this subsection, we study the monotonicity of $\la_1(d,q,\bm r)$ with respect to the advection rate $q$.

\begin{lemma}\label{prop-p}
Suppose that  \textbf{H}1 or \textbf{H}1$^{*}$ holds. Let $\la_1(d,q,\bm r)$ be the principal eigenvalue of \eqref{eigen}. Then for fixed $d>0$, $\la_1(d,q,\bm r)$ is strictly decreasing with respect to $q$ in $[0,\infty)$. Moreover,
\begin{equation}
\lim_{q\to0}\la_1(d,q,\bm r)=\la_1(d,0,\bm r) \text{  and  }
\lim_{q\to\infty}\la_1(d,q,\bm r)=-\infty.
\end{equation}
\end{lemma}
\begin{proof}
Let $\bm \phi=(\phi_1,\phi_2,\dots,\phi_n)^T\gg \bm 0$ be the  eigenvector corresponding to the principal eigenvalue $\la_1(d,q,\bm r)$ with
\begin{equation}\label{normalphi}
\sum_{i=1}^n \phi_i=1, \;\;q\in[0,\infty).
\end{equation}
Differentiating \eqref{eigen} with respect to $q$ yields
\begin{equation}\label{sp2}
\frac{\partial\lambda_1}{\partial q}\phi_i + \lambda_1 \frac{\partial\phi_i}{\partial q}= \sum_{j=1}^n\left(dD_{ij}+qQ_{ij}\right) \frac{\partial\phi_j}{\partial q} +\sum_{j=1}^n Q_{ij}\phi_j+r_i\frac{\partial\phi_i}{\partial q}.
\end{equation}
Then multiplying \eqref{sp2} by $\phi_i$ and \eqref{eigen} by $\frac{\partial\phi_i}{\partial q}$ and taking the difference, we have
\begin{equation}\label{sp3}
\frac{\partial\lambda_1}{\partial q}\phi_i^2=\sum_{j\ne i}\left(dD_{ij}+qQ_{ij}\right)\left(\frac{\partial\phi_j}{\partial q}\phi_i-\frac{\partial\phi_i}{\partial q}\phi_j\right)+\sum_{j=1}^nQ_{ij}\phi_i\phi_j.\\
\end{equation}
Let
 \begin{equation*}\label{alpha}
(\beta_1, \beta_2, \beta_3,\dots, \beta_n)=\left(1,\ds\frac{d}{d+q},\left(\ds\frac{d}{d+q}\right)^2,\dots, \left(\ds\frac{d}{d+q}\right)^{n-1}\right).
\end{equation*}
Multiplying \eqref{sp3} by ${\beta}_i$ and summing them over $i$, we obtain
\begin{equation}\label{sp5}
\frac{\partial\lambda_1}{\partial q} \sum_{i=1}^n {\beta}_i \phi_i^2 = \sum_{i=1}^n \sum_{j\ne i} \beta_i\left(dD_{ij}+qQ_{ij}\right)\left(\frac{\partial\phi_j}{\partial q}\phi_i-\frac{\partial\phi_i}{\partial q}\phi_j\right)
+\sum_{i=1}^n\sum_{j=1}^n\beta_iQ_{ij}\phi_i\phi_j.
\end{equation}
A direct computation yields
\begin{equation}\label{sp6}
\begin{split}
&\sum_{i=1}^n \sum_{j\ne i} \beta_i\left(dD_{ij}+qQ_{ij}\right)\left(\frac{\partial\phi_j}{\partial q}\phi_i-\frac{\partial\phi_i}{\partial q}\phi_j\right)\\
=&\sum_{i=1}^{n-1}\left[\beta_id\left(\frac{\partial\phi_{i+1}}{\partial q}\phi_i-\frac{\partial\phi_i}{\partial q}\phi_{i+1}\right)+\beta_{i+1}(d+q)\left(\frac{\partial\phi_i}{\partial q}\phi_{i+1}-\frac{\partial\phi_{i+1}}{\partial q}\phi_i\right)\right]\\
=&\sum_{i=1}^{n-1}\left[(\beta_id -\beta_{i+1}(d+q))\left(\frac{\partial\phi_{i+1}}{\partial q}\phi_i-\frac{\partial\phi_i}{\partial q}\phi_{i+1}\right)\right]=0,
\end{split}
\end{equation}
where we have used $\beta_id -\beta_{i+1}(d+q)=0$ for all $i=1, 2,\dots,n-1$.
This, combined with \eqref{sp5}, implies that
\begin{equation}\label{sp7}
\begin{split}
&\frac{\partial\lambda_1}{\partial q} \sum_{i=1}^n {\beta}_i \phi_i^2=\sum_{i=1}^n\sum_{j=1}^n\beta_iQ_{ij}\phi_i\phi_j\\=&-\sum_{i=1}^n\beta_i\phi_i^2
+\sum_{i=1}^{n-1}\beta_{i+1}\phi_i\phi_{i+1}\\
=& -\frac{\beta_1}{2}\phi^2_1-\frac{\beta_n}{2}\phi^2_n-\sum_{i=1}^{n-1}\left(\frac{\beta_i}{2}\phi_i^2-\beta_{i+1}\phi_i\phi_{i+1}+\frac{\beta_{i+1}}{2}\phi_{i+1}^2\right)<0,
\end{split}
\end{equation}
where we have used the fact that
$$
\frac{\beta_i}{2}\phi_i^2-\beta_{i+1}\phi_i\phi_{i+1}+\frac{\beta_{i+1}}{2}\phi_{i+1}^2\ge\frac{\beta_{i+1}}{2}(\phi_i-\phi_{i+1})^2  \ge0
$$
as $\beta_{i}\ge\beta_{i+1}$ for $i=1,\dots,n-1$.
This implies that $\la_1(d,q,\bm r)$ is strictly decreasing with respect to $q$ in $[0,\infty)$.

Clearly, $\ds\lim_{q\to0}\la_1(d,q,\bm r)=\la_1(d,0,\bm r)$. It remains to show $\ds\lim_{q\rightarrow\infty}\la_1(d,q,\bm r)=-\infty$.
Since $\la_1(d,q,\bm r)$ is decreasing in  $q$,  the limit $\ds\lim_{q\to\infty}\la_1(d,q,\bm r)$ exists in $[-\infty, \infty)$. Suppose to the contrary that
$\ds\lim_{q\to\infty}\la_1(d,q,\bm r)\in(-\infty,\infty)$. By \eqref{normalphi}, up to a subsequence, we have
 $\ds\lim_{q\to\infty}\bm \phi=\bm \phi^*$, where $\bm \phi^*=(\phi^*_1,\dots,\phi^*_n)^T\ge\bm 0$ and $\sum_{i=1}^n\phi_i^*=1$.
Dividing both sides of \eqref{eigen} by $q$ and taking $q\to\infty$, we have
$$
\sum_{j=1}^n Q_{ij} \phi^*_i=\bm 0, \ i=1,\dots, n.
$$
This implies that $\bm \phi^*=\bm 0$, which is a contradiction. Therefore, we have $\ds\lim_{q\to \infty}\la_1(d,q,\bm r)=-\infty$.
\end{proof}

From  Lemmas \ref{DS-single}-\ref{theorem_quasi} and \ref{prop-p}, we obtain the following results about the impact of $d$ and $q$ on the dynamics of model \eqref{pat-s} for case (a).
\begin{proposition}\label{effp}
Suppose that \textbf{H}1 holds. Then the following statements hold:
\begin{enumerate}
\item [{\rm (i)}] If $\ds\sum_{i=1}^n r_i> 0$, then for any $d>0$ there exists $q_{\bm r}^*(d)>0$ such that
$\la_1(d,q_{\bm r}^*(d), \bm r)=0$, $\la_1(d,q, \bm r)<0$ for $q>q_{\bm r}^*(d)$, and $\la_1(d,q, \bm r)>0$ for $q<q_{\bm r}^*(d)$;
Moreover, we have the following results:
\begin{enumerate}
\item [{$\rm (i_1)$}] If $q\ge q_{\bm r}^*(d)$, then the trivial equilibrium $\mathbf{0}$ of model \eqref{pat-s} is globally asymptotically stable;
\item [{$\rm (i_2)$}] If $q<q_{\bm r}^*(d)$, model \eqref{pat-s} admits a unique positive equilibrium, which is globally asymptotically stable;
\end{enumerate}
\item [{\rm (ii)}] If $\ds\sum_{i=1}^n r_i<0<\max_{1\le i\le n} r_i$, then there exists  $d^*>0$ such that $\la_1(d^*,0,\bm r)=0$,  $\la_1(d,0,\bm r)<0$ for $d>d^*$, and $\la_1(d,0,\bm r)>0$ for $d<d^*$; Moreover, we have the following results:
     \begin{enumerate}
\item [{$\rm (ii_1)$}] If $d\in(0,d^*)$, then there exists $q_{\bm r}^*(d)>0$ such that $(i_1)$-$(i_2)$ hold;
\item [{$\rm (ii_2)$}] If $d\ge d^*$, then for any $q>0$, the trivial equilibrium $\mathbf{0}$ of model \eqref{pat-s} is globally asymptotically stable;
\end{enumerate}
\item [{\rm (iii)}] If $\ds\max_{1\le i\le n} r_i\le 0$, then the trivial equilibrium $\mathbf{0}$ of model \eqref{pat-s} is globally asymptotically stable for any $d>0$ and $q\ge 0$.
\end{enumerate}
\end{proposition}

\begin{proof}
Note that $D$ is an  irreducible essentially nonnegative matrix with $s(D)=0$ corresponding with a positive eigenvector $(1/n,\dots, 1/n)$. It follows from Lemma \ref{theorem_quasi} that
$$\frac{\partial \la_1(d,0,\bm r)}{\partial d}\le0,
$$
where the equality holds if and only $r_1=\dots=r_n$. Moreover, $$
\lim_{d\rightarrow 0}\la_1(d,0,\bm r)=\max_{1\le i\le n} r_i \;\;
\text{and} \;\; \lim_{d \rightarrow\infty}\la_1(d,0,\bm r)=\ds\f{\sum_{i=1}^n{r_i}}{n}.
$$
This, combined with Lemmas \ref{DS-single} and \ref{prop-p}, implies  $(i)$-$(iii)$.
\end{proof}

A similar result holds for case (b).
\begin{proposition}\label{qb}
Suppose that \textbf{H}1$^{*}$ holds. Then we the following results:
\begin{enumerate}
\item [{\rm (i)}] If $\ds\max_{1\le i\le n} r_i>0$, then there exists  $d^*>0$ such that $\la_1(d^*,0,\bm r)=0$,  $\la_1(d,0,\bm r)<0$ for $d>d^*$, and $\la_1(d,0,\bm r)>0$ for $d<d^*$; Moreover, we have:
     \begin{enumerate}
\item [{$\rm (i_1)$}] If $d\in(0,d^*)$, then there exists $q_{\bm r}^*(d)>0$ such that $(i_1)$-$(i_2)$ in Proposition \ref{effp} hold;
\item [{$\rm (i_2)$}] If $d\ge d^*$, then for any $q>0$, the trivial equilibrium $\mathbf{0}$ of model \eqref{pat-s} is globally asymptotically stable;
\end{enumerate}
\item [{\rm (ii)}] If $\ds\max_{1\le i\le n} r_i\le 0$, then the trivial equilibrium $\mathbf{0}$ of model \eqref{pat-s} is globally asymptotically stable  for any $d>0$ and $q\ge 0$.
\end{enumerate}
\end{proposition}

\begin{proof}
Since $s(D)<0$, by Lemma \ref{theorem_quasi}, we have
$$\frac{\partial \la_1(d,0,\bm r)}{\partial d}<0,
$$
and
 $$
\lim_{d\rightarrow 0}\la_1(d,0,\bm r)=\max_{1\le i\le n} r_i \;\;
\text{and} \;\; \lim_{d \rightarrow\infty}\la_1(d,0,\bm r)=-\infty.
$$
This, combined with Lemmas \ref{DS-single} and \ref{prop-p}, implies  {$\rm (i)$}-{$\rm (ii)$}.
\end{proof}

\subsection{Dependence of $\la_1(d,q,\bm r)$ on $d$}
In this section, we study the dependence of $\la_1(d,q,\bm r)$ on $d$  for  cases (a)-(b). When the directed movement rate $q=0$, we know that $\lambda_1(d, 0, \bm r)$ is decreasing in $d\in (0, \infty)$. However, this may no longer be true when $q>0$.

We first compute the limits of $\lambda_1$ as $d\rightarrow 0$ or $\infty$ in case (a).
\begin{lemma}\label{consr}
Suppose that  \textbf{H}1 holds. Let $\la_1(d,q,\bm r)$ be the principal eigenvalue of \eqref{eigen}. Then, we have the following:
 $$
 \lim_{d\to0}\la_1(d,q,\bm r)=\ds\max_{1\le i\le n}r_i-q, \;\;\text{and}\;\;\lim_{d\to\infty}\la_1(d,q,\bm r)=\ds\frac{\sum_{i=1}^nr_i-q}{n}.
 $$
\end{lemma}
\begin{proof}
Firstly, it is easy to see that $\ds\lim_{d\to0}\la_1(d,q,\bm r)=\la_1(0, q, \bm r)=\ds\max_{1\le i\le n}r_i-q$. Then we compute the limit of  $ \la_1(d,q,\bm r)$ as $q\rightarrow \infty$. Let $\bm \phi=(\phi_1,\phi_2,\ldots,\phi_n)^T\gg \bm 0$ be the eigenvector corresponding to the principal eigenvalue $\la_1(d,q,\bm r)$ with $\sum_{i=1}^n\phi_i=1$. Summing all the equations in \eqref{eigen}, we have
\begin{equation}\label{sum}
\sum_{i=1}^n\sum_{j=1}^n (d D_{ij}+qQ_{ij})\phi_j+\sum_{i=1}^n r_i\phi_i=\la_1(d,q,\bm r) \sum_{i=1}^n \phi_i.
\end{equation}
It follows from \textbf{H}1 that
$$
\sum_{i=1}^n\sum_{j=1}^n D_{ij} \phi_j=0 \ \ \text{and} \ \  \sum_{i=1}^n\sum_{j=1}^n Q_{ij} \phi_j=-\phi_n.
$$
Therefore, by \eqref{sum}, we have
\begin{equation}\label{bound}
-q\phi_n+\sum_{i=1}^n r_i\phi_i=\la_1(d,q,\bm r)\sum_{i=1}^n \phi_i .
\end{equation}
This gives a bound for $\la_1(d,q,\bm r)$:
$$
\min_{1\le i\le n} r_i-q \le \la_1(d,q,\bm r)\le \max_{1\le i\le n} r_i,
$$
which  implies that
\begin{equation}\label{limscona}
\lim_{d\to\infty}\la_1(d,q,\bm r)\in(-\infty,\max_{1\le i\le n} r_i].
\end{equation}
Up to a subsequence, we may assume $\ds\lim_{d\rightarrow\infty} \la_1(d, q, \bm r)=a$ and  $\ds\lim_{d\to\infty}\bm\phi=\bm{\bar\phi}$, where $\bm{
\bar\phi}=(\bar \phi_1,\dots,\bar\phi_n)^T\ge\bm 0$ and $\sum_{i=1}^n\bar\phi_i=1$.
Dividing both sides of \eqref{eigen} by $d$ and taking $d\to\infty$, we have
$D\bm{\bar\phi}=\bm 0$, which implies that
\begin{equation}\label{conb}
\bm{\bar\phi}=(\bar\phi_1,\dots,\bar\phi_n)^T=\left(\frac{1}{n},\dots,\frac{1}{n}\right)^T.
\end{equation}
Taking $d\rightarrow\infty$ in \eqref{bound}, we have
$$
-q\bar\phi_n+\sum_{i=1}^n r_i\bar\phi_i=a \sum_{i=1}^n \bar\phi_i.
$$
This gives
 $$\lim_{d\to\infty}\la_1(d,q,\bm r)= a=\ds\frac{\sum_{i=1}^nr_i-q}{n}.$$
\end{proof}

For $\hat{\bm r}\gg\bm 0$, the principal eigenvalue $\la_1(d,q, \hat{\bm r})$ satisfies the following property for case $(a)$, which will be useful later.
\begin{lemma}\label{consr2}
Suppose that \textbf{H}1 holds. Let $\la_1(d,q, \hat{\bm r})$ be the principal eigenvalue of \eqref{eigen}  with $\hat{\bm r}\gg\bm 0$. If
$\la_1(d^*,q, \hat{\bm r})=0$ for some $d^*>0$, then
\begin{equation}\label{lage}
 \f{\partial}{\partial d} \la_1(d, q, \hat{\bm r})\big|_{d=d^*}>0.
\end{equation}
\end{lemma}
\begin{proof}
Let $\bm \phi=(\phi_1,\phi_2,\dots,\phi_n)^T\gg \bm 0$ be the positive eigenvector corresponding to the  eigenvalue $\la_1(d,q, \hat{\bm r})$ with $\sum_{i=1}^n\phi_i=1$.
By similar arguments as in the proof of Lemma~\ref{prop-p},
we obtain
\begin{equation*}
\frac{\partial\lambda_1}{\partial d} \sum_{i=1}^n {\beta}_i \phi_i^2 = \sum_{i=1}^n\sum_{j=1}^n\beta_iD_{ij}\phi_i\phi_j,
\end{equation*}
where $\beta_i$ is defined in \eqref{alpha}.
A direct computation implies that
\begin{equation}\label{mp5b}
\frac{\partial\lambda_1}{\partial d} \ds\sum_{i=1}^n {\beta}_i \phi_i^2 = \ds\sum_{i=1}^{n-1}\beta_i\left(\phi_{i+1}-\phi_i\right)\left[\phi_{i}-\left(\ds\frac{d}{d+q}\right)\phi_{i+1}\right].
\end{equation}
If $\la_1(d^*,q, \hat{\bm r})=0$, then we see from \eqref{eigen} that
\begin{equation}\label{requiv0}
\begin{split}
&(d^*+q)(\phi_{n-1}-\phi_n)=-\hat r_n\phi_n,\\
&(d^*+q)(\phi_{i-1}-\phi_i)=-\hat r_i\phi_i+d^*(\phi_{i}-\phi_{i+1}),\;\;i=2,\dots,n-1.
\end{split}
\end{equation}
Since $\hat{\bm r}\gg 0$, we have
\begin{equation}\label{order}
\phi_{1}<\phi_2<\dots<\phi_n.
\end{equation}
 Summing up the first $k$ equations in \eqref{eigen}, where $1\le k\le n-1$, we find
$$
d^*\phi_{k+1}-(d^*+q)\phi_k=-\sum_{i=1}^k \hat r_i\phi_i<0.
$$
This, combined with \eqref{mp5b} and \eqref{order}, implies \eqref{lage}.
\end{proof}

We also have the limits of $\lambda_1$ as $d\rightarrow 0$ or $\infty$ in case (b).
\begin{lemma}\label{consrr}
Suppose that  \textbf{H}1$^*$ holds. Let $\la_1(d,q,\bm r)$ be the principal eigenvalue of \eqref{eigen}. Then, we have the following:
 $$
 \lim_{d\to0}\la_1(d,q,\bm r)=\ds\max_{1\le i\le n}r_i-q, \;\;\text{and}\;\;\lim_{d\to\infty}\la_1(d,q,\bm r)=-\infty.
 $$
\end{lemma}
\begin{proof}
The proof is similar to that of Lemma \ref{consr}, and the  difference is that \eqref{bound} is replaced by the following equation:
\begin{equation}\label{bound2}
-(d+q)\phi_n+\sum_{i=1}^n r_i\phi_i=\la_1(d,q,\bm r) \sum_{i=1}^n \phi_i.
\end{equation}
This gives a bound for $\la_1(d,q,\bm r)$:
$$
\min_{1\le i\le n} r_i-(d+q)\le \la_1(d,q,\bm r)\le \max_{1\le i\le n} r_i.
$$
Assume to the contrary that $\ds\lim_{d\rightarrow\infty} \la_1(d, q, \bm r) \neq-\infty$. Up to a subsequence, we may assume $\ds\lim_{d\rightarrow\infty} \la_1(d, q, \bm r)=a\in (-\infty, \max_{1\le i\le n} r_i]$ and  $\ds\lim_{d\to\infty}\bm\phi=\bm{ \bar\phi}$, where $\bm{ \bar\phi}=(\bar \phi_1,\dots,\bar\phi_n)\ge\bm 0$ and $\sum_{i=1}^n\bar\phi_i=1$.
Dividing both sides of \eqref{eigen} by $d$ and taking $d\to\infty$, we have
$D\bm{\bar\phi}=\bm 0$. So $\bm{\bar\phi}$ is a nonnegative eigenvector corresponding with eigenvalue 0 of $D$. This contradicts with $s(D)<0$. Therefore, $\ds \lim_{d\rightarrow\infty} \la_1(d, q, \bm r) =-\infty$.
\end{proof}

\subsection{Some properties on $q_{\bm r}^*(d)$}
In this subsection, we give some properties on function $q_{\bm r}^*(d)$ obtained in Propositions~\ref{effp} and \ref{qb}, which will be used in the next section.

We first consider case $(a)$.
\begin{lemma}\label{qstr}
Suppose that \textbf{H}1 holds and ${\bm r}\gg\bm 0$, and let $q^*_{\bm r}(d)$ be defined in Proposition \ref{effp}. Then the following statements about $q_{{\bm r}}^*(d)$ hold:
\begin{enumerate}
\item [{$\rm (i)$}]
$q_{{\bm r}}^*(d)$ is strictly increasing with respect to $d$ in $(0,\infty)$;
\item [{$\rm (ii)$}]
$\ds\lim_{d\to0}q_{{\bm r}}^*(d)=\max_{1\le i\le n} r_i$, and $\ds\lim_{d\to\infty}q_{{\bm r}}^*(d)=\ds\sum_{i=1}^n r_i$;

\item [{$\rm (iii)$}] If $\bm { r_1}> \bm { r_2}\gg \bm 0$, then $q^*_{ \bm { r_1}}(d)>q^*_{\bm { r_2}}(d)$ for any  $d>0$.
\end{enumerate}
\end{lemma}
\begin{proof}
(i) Let $d_1>d_2>0$. Then, by the definition of $q^*_{{\bm r}}(d)$, we have
  $$
  \la_1\left(d_1,q_{{\bm r}}^*(d_1), {\bm r}\right)=\la_1\left(d_2,q_{{\bm r}}^*(d_2),\bm{ r}\right)=0.
  $$
This, combined with Lemma \ref{consr2}, yields
\begin{equation}\label{qdr2}
\la_1\left(d_1,q_{{\bm r}}^*(d_1), {\bm r}\right)-\la_1\left(d_1,q_{{\bm r}}^*(d_2),{\bm r}\right)=\la_1\left(d_2,q_{{\bm r}}^*(d_2),{\bm r}\right)-\la_1\left(d_1,q_{{\bm r}}^*(d_2),{\bm r}\right)<0.
\end{equation}
By Lemma \ref{prop-p}, $\la_1(d, q, {\bm r})$ is strictly decreasing with respect to $q$. Therefore, we can see from \eqref{qdr2} that $q_{{\bm r}}^*(d_1)>q_{{\bm r}}^*(d_2)$.

 (ii) Since $q_{{\bm r}}^*(d)$ is strictly increasing with respect to $d$, the limit $\ds\lim_{d\to0} q_{\bm r}^*(d)$ exists in $[0, \infty)$ and the limit $\ds\lim_{d\to\infty} q_{\bm r}^*(d)$ exists in $(0, \infty]$. We denote
\begin{equation*}
a_1:=\lim_{d\to0} q_{{\bm r}}^*(d)\;\;\text{and}\;\; a_2:=\lim_{d\to\infty} q_{{\bm r}}^*(d).
\end{equation*}
We first suppose that $a_1\ne0$ (i.e. $a_1>0$).
Then, for sufficiently small $\epsilon>0$, there exists $\bar d>0$ such that $0<a_1-\epsilon<q_{{\bm r}}^*(d)<a_1+\epsilon$ for all $0<d<\bar d$.  Since $\la_1(d, q, {\bm r})$ is strictly decreasing in $q$, we have
\begin{equation}\label{qbound}
\la_1(d, a_1+\epsilon,  {\bm r})<\la_1(d, q_{{\bm r}}^*(d), {\bm r})=0\le \la_1(d, a_1-\epsilon, {\bm r})
\end{equation}
for all  $0<d<\bar d$. Taking $d\to 0$ in \eqref{qbound} and using Lemma \ref{consr}, we have
\begin{equation}\label{supp}
\max_{1\le i\le n}  r_i -(a_1+\epsilon)\le 0 \le \max_{1\le i\le n}  r_i -(a_1-\epsilon).
\end{equation}
This gives $\ds\max_{1\le i\le n}  r_i-\epsilon\le a_1\le \ds\max_{1\le i\le n}  r_i+\epsilon$. Since $\epsilon>0$ is arbitrary, we have $a_1=\max_{1\le i\le n}  r_i$. If $a_1=0$, then the first inequality of \eqref{supp} still holds, which gives
$\ds\max_{1\le i\le n}  r_i\le \epsilon$. Since $\epsilon>0$ is arbitrary, we have
${\bm r}=\bm0$, which is a contradiction. This proves $\ds\lim_{d\to0} q_{\bm r}^*(d)=\ds\max_{1\le i\le n} r_i$.

Let $\bm \phi=(\phi_1,\phi_2,\dots,\phi_n)^T\gg \bm 0$ be the  eigenvector corresponding to the  eigenvalue $\la_1(d,q_{{\bm r}}^*(d), {\bm r})=0$ with $\sum_{i=1}^n\phi_i=1$. Then, we have
\begin{equation}\label{qdr3}
dD\bm\phi+q_{{\bm r}}^*(d)Q\bm\phi+\text{diag}( r_i)\bm \phi=\bm 0.
\end{equation}
Up to a subsequence, we may assume
$\ds\lim_{d\to\infty}\bm \phi=\bm \phi^*$, where $\bm \phi^*=(\phi^*_{1},\dots,\phi^*_{n})^T\ge\bm 0$ and $\sum_{i=1}^n\phi_{i}^*=1$.
Now we claim that $a_2\ne\infty$. Suppose to the contrary that $a_2=\infty$.
Multiplying \eqref{qdr3} by $(1,\dots,1)$ and dividing both sides by $q_{\bm r}^*(d)$, we obtain
\begin{equation}\label{qdr6}
-\phi_n+\ds\f{1}{q_{{\bm r}}^*(d)}\sum_{i=1}^n r_i\phi_i=0,
\end{equation}
which yields $\phi^*_{n}=\ds\lim_{d\to\infty}\phi_{n}=0$.
By virtue of \eqref{qdr3} again, we obtain that
\begin{equation}\label{qdr5}
\begin{split}
&\phi_{n-1}-\phi_{n}=\frac{-r_n\phi_n}{d+q_{{\bm r}}^*(d)},\\
&\phi_{i-1}-\phi_{i}=\frac{d}{{d+q_{{\bm r}}^*(d)}}(\phi_{i}-\phi_{i+1})-\frac{r_i\phi_i}{d+q_{{\bm r}}^*(d)},\;\;i=2,\dots,n-1.\\
\end{split}
\end{equation}
Taking $d\to\infty$ in \eqref{qdr5}, we have $\phi^*_{1}=\dots=\phi^*_{n}$, and consequently $\bm\phi^*=\bm 0$. This is a contradiction, and hence $a_2\in(-\infty,\infty)$.
Dividing \eqref{qdr3} by $d$ and taking $d\to\infty$, we have
$D\bm\phi^*=\bm 0$, which implies that $$\bm \phi^*=(\phi_{1}^*,\dots,\phi_{n}^*)^T=\left(\frac{1}{n},\dots,\frac{1}{n}\right)^T.$$
Then taking the limit of \eqref{qdr6}, we have $a_2=\sum_{i=1}^n r_i$.

(iii) Clearly, if $\bm { r_1}>\bm { r_2}$, then
\begin{equation}\label{qr1}
\la_1(d,q,\bm { r_1})>\la_1(d,q,\bm { r_2})\;\;\text{for any}\;\;d,q>0.
\end{equation}
Note that $\la_1(d,q_{\bm { r_1}}^*(d),\bm { r_1})=\la_1(d,q_{\bm {r_2}}^*(d),\bm { r_2})=0$. This, combined with \eqref{qr1}, implies that
$$\la_1(d,q_{\bm { r_1}}^*(d),\bm { r_1})-\la_1(d,q_{\bm { r_2}}^*(d),\bm { r_1})=\la_1(d,q_{\bm {r_2}}^*(d),\bm { r_2})-\la_1(d,q_{\bm { r_2}}^*(d),\bm { r_1})<0.$$
Since $\la_1(d, q,\bm { r})$ is strictly decreasing in $q$, we have $q^*_{\bm { r_1}}(d)>q^*_{\bm { r_2}}(d)$ as desired.
\end{proof}

Next we consider case $(b)$.

\begin{lemma}\label{qstr2}
Suppose that \textbf{H}1$^*$ holds, and let $q_{\bm r}^*(d)$ and $d^*$ be defined in Proposition \ref{qb}. Then the following statements hold:
\begin{enumerate}
\item [{$\rm (i)$}] If ${\bm r}=(r_1,\dots,  r_n)$ satisfies $\ds\max_{1\le i\le n} r_i>0$, then
$$
\lim_{d\to0}q_{{\bm r}}^*(d)=\max_{1\le i\le n}r_i\;\; \text{and}\;\; \lim_{d\to {d^*}}q_{{\bm r}}^*(d)=0;
$$
\item [{$\rm (ii)$}] If $\bm { r_j}=( r_{j, 1}, \dots,  r_{j, n})$ satisfies $\bm { r_1}> \bm { r_2}$ and $\ds\max_{1\le i\le n} r_{j, i}>0$ for $j=1, 2$, then $q^*_{ \bm { r_1}}(d)>q^*_{\bm { r_2}}(d)$ for any  $d\in (0, \hat d)$. Here, $\hat d=\min\{{d_1^*}, {d_2^*}\}= {d_2^*}$, where $\la_1({d_1^*}, 0, \bm { r_1})=\la_1({d_2^*}, 0, \bm {r_2})=0$;

\item [{$\rm (iii)$}] If $\bm {r}=( r, \dots,  r)$ with $r>0$ and $q^*_{\bm r}(\bar d)<r$ for some $\bar d\in (0, d^*)$, then $[q^*_{\bm r}(d)]'<0$ for $d\in [\bar d, d^*)$.
\end{enumerate}
\end{lemma}
\begin{proof}
(i) By  Proposition \ref{qb} and $\ds\max_{1\le i\le n} r_i>0$, $q_{{\bm r}}^*(d)$ is well-defined.  Up to a subsequence, we have
\begin{equation*}
a_1:=\lim_{d\to0} q_{{\bm r}}^*(d)\;\;\text{and}\;\; a_2:=\lim_{d\to {d^*}} q_{{\bm r}}^*(d),
\end{equation*}
where $a_1, a_2\in [0, \infty]$. Using the same argument as in the proof of Lemma \ref{qstr}, we can show $a_1, a_2\neq \infty$. Then similar to Lemma \ref{qstr}, we can use Lemma \ref{consrr} to compute $a_1=\ds\max_{1\le i\le n}  r_i$.

Now we claim that $a_2=0$. If it is not true, then for sufficiently small $\epsilon>0$, there exists $\bar d>0$ such that $0<a_2-\epsilon<q_{{\bm r}}^*(d)<a_2+\epsilon$ for all $d\in (\bar d,   {d^*})$. Since $\la_1(d, q_{{\bm r}}^*(d), {\bm r})=0$ and $\la_1(d, q, {\bm r})$ is strictly decreasing in $q$, we have
$$
\la_1(d,  a_2+\epsilon, {\bm r})<\la_1(d, q_{{\bm r}}^*(d),  {\bm r})=0<\la_1(d,  a_2-\epsilon, {\bm r}).
$$
for all $d\in (\bar d,  {d^*})$. Taking $d\to {d^*}$, we have
$$
\la_1({d^*},  a_2+\epsilon, {\bm r})\le 0\le \la_1({d^*},  a_2-\epsilon, {\bm r}).
$$
Taking $\epsilon\to 0$, we have $\la_1({d^*},  a_2, {\bm r})= 0$, which contradicts with $a_2>0$.

The proof of (ii) is similar to the one for Lemma \ref{qstr}, so we omit it here.

(iii) Let $\bm \phi=(\phi_1,\phi_2,\dots,\phi_n)^T\gg \bm 0$ be the  eigenvector corresponding to the  eigenvalue $\la_1(d,q_{{\bm r}}^*(d), {\bm r})=0$ with $\sum_{i=1}^n\phi_i=1$. Then, we have
\begin{equation}\label{qdphi}
d\sum_{j=1}^n D_{ij}\phi_j+q_{{\bm r}}^*(d)\sum_{j=1}^n Q_{ij}\phi_j+r\phi_i=0.
\end{equation}
Differentiating \eqref{qdphi} with respect to $d$, we obtain
\begin{equation}\label{qdphi2}
\sum_{j=1}^n D_{ij}\phi_j+\sum_{j=1}^n D_{ij}\phi_j'+[q_{{\bm r}}^*(d)]'\sum_{j=1}^n Q_{ij}\phi_j+q_{{\bm r}}^*(d)\sum_{j=1}^n Q_{ij}\phi_j'+r\phi_i'=0.
\end{equation}
Multiplying \eqref{qdphi2} by $\phi_i$ and \eqref{qdphi} by $\phi_i'$ and taking the difference, we have
\begin{equation}\label{qdphi3}
    [q_{{\bm r}}^*(d)]'\sum_{j=1}^n Q_{ij}\phi_i\phi_j=-\sum_{j=1}^n (dD_{ij}+q_{{\bm r}}^*(d)Q_{ij})(\phi_i\phi'_j-\phi'_i\phi_j)-\sum_{j=1}^n D_{ij}\phi_i\phi_j.
\end{equation}
Similar to the proof of Lemma \ref{prop-p}, let
 \begin{equation*}
(\beta_1, \beta_2, \beta_3,\dots, \beta_n)=\left(1,\ds\frac{d}{d+q_{{\bm r}}^*(d)},\left(\ds\frac{d}{d+q_{{\bm r}}^*(d)}\right)^2,\dots, \left(\ds\frac{d}{d+q_{{\bm r}}^*(d)}\right)^{n-1}\right).
\end{equation*}
Multiplying \eqref{qdphi3} by ${\beta}_i$ and summing them over $i$, we obtain
\begin{equation}\label{qdphi4}
[q_{{\bm r}}^*(d)]' \sum_{i=1}^n\sum_{j=1}^n\beta_iQ_{ij}\phi_i\phi_j =  -\sum_{i=1}^n\sum_{j=1}^n\beta_iD_{ij}\phi_i\phi_j,
\end{equation}
where we have used
$$\sum_{i=1}^n \sum_{j=1}^n \beta_i\left(dD_{ij}+q_{{\bm r}}^*(d)Q_{ij}\right)(\phi_i\phi'_j-\phi'_i\phi_j)=0.
$$
By \eqref{sp7}, we have
\begin{equation}\label{qdphi5}
\sum_{i=1}^n\sum_{j=1}^n\beta_iQ_{ij}\phi_i\phi_j<0.
\end{equation}
A direct computation gives
\begin{equation}\label{qphi}
\sum_{i=1}^n\sum_{j=1}^n\beta_iD_{ij}\phi_i\phi_j=\ds\sum_{i=1}^{n-1}\beta_i\left(\phi_{i+1}-\phi_i\right)\left[\phi_{i}-\left(\ds\frac{d}{d+q_{{\bm r}}^*(d)}\right)\phi_{i+1}\right]-\beta_n\phi_n^2.
\end{equation}

Suppose  $q_{{\bm r}}^*(\bar d)<r$ for some $\bar d\in (0, d^*)$. We can rewrite \eqref{qdphi} as
\begin{equation}\label{phidd}
\begin{split}
&\bar d(\phi_1-\phi_2)=(r-q_{{\bm r}}^*(\bar d))\phi_1,\\
&\bar d(\phi_{i}-\phi_{i+1})=(\bar d+q_{{\bm r}}^*(\bar d))(\phi_{i-1}-\phi_i)+r\phi_i,\;\;i=2,\dots,n-1,\\
&\bar d(\phi_{n-1}-2\phi_n)+q_{{\bm r}}^*(\bar d)(\phi_{n-1}-\phi_n)+r\phi_n=0.
\end{split}
\end{equation}
It follows from \eqref{phidd} that $\phi_1>\dots>\phi_n$. So, by \eqref{qphi}, we have
$$
\sum_{i=1}^n\sum_{j=1}^n\beta_iD_{ij}\phi_i\phi_j<0.
$$
This combined with \eqref{qdphi4}-\eqref{qdphi5} gives $[q_{{\bm r}}^*(\bar d)]'<0$. Therefore, we must have $[q_{{\bm r}}^*(d)]'<0$ for $d\in [\bar d, d^*)$.
\end{proof}

\begin{remark}
 The monotonicity of $\lambda_1(d, q, \bm r)$ in $q$ for cases (a) and (b) of reaction-diffusion models was proved in \cite{lou2014evolution}. To our best knowledge,
 the properties of $q^*_{\bm r}(d)$ were not studied for reaction-diffusion models.

\end{remark}

\section{Invasion analysis for two competing species}

In this section, we study the evolution of diffusion and advection rates by considering the two species competition model \eqref{pat-cp}. Throughout this section, assume $\bm r=(r,\dots, r)\gg \bm~0$.

If we treat  $\bm u=(u_1,\dots, u_n)^T$ as the resident species and  $\bm v=(v_1,\dots,v_n)^T$ as the mutating/invading species,  our purpose is to find conditions under which $\bm v$ can or cannot invade. To this aim, we suppose that species $\bm u$ has been established, and model \eqref{pat-cp} has a unique semi-trivial equilibrium by $(\bm u^*,\bm 0)$, where $\bm u^*=(u_1^*,\dots,u_n^*)^T\gg\bm 0$ satisfies
\begin{equation}\label{semitu}
\ds\sum_{j=1}^{n}(d_1D_{ij}+q_1Q_{ij})u_j+u_i(r-u_i)=0, \;\;i=1,\dots,n,
\end{equation}
where $r>0$  in this section.   Then we study the stability of $(\bm u^*, 0)$  when $d_2$ and $q_2$ are different from $d_1$ and $q_1$, respectively. Biologically, if $(\bm u^*, 0)$ is stable, this means that an introduction of small amount of species $\bm v$ cannot invade species $\bm u$; if $(\bm u^*, 0)$ is unstable, this means that a small amount of species $\bm v$ may be able to destabilize the system and the invading/mutating species  $\bm v$ may be established.

We denote the $\bm v-$only semi-trivial equilibrium by $(\bm 0,\bm v^*)$ if it exists, where
$\bm v^*=(v_1^*,\dots,v_n^*)^T\gg\bm 0$ solves
\begin{equation}\label{semitv}
\ds\sum_{j=1}^{n}(d_2D_{ij}+q_2Q_{ij})v_j+v_i(r-v_i)=0, \;\;i=1,\dots,n.
\end{equation}

\subsection{Invasion analysis for case (a)}
By Proposition \ref{effp}, $q_{\bm r}^*(d_1)>0$ exists for any $d_1>0$.
We suppose that species $\bm{u}$ is established,  i.e.,
\begin{enumerate}
\item[\textbf{H}2.] $q_1<q_{\bm r}^*(d_1)$,
\end{enumerate}
where $q_{\bm r}^*(d_1)$ satisfies $\la_1(d_1, q_{\bm r}^*(d_1), \bm r)=0$. If \textbf{H}2 is satisfies,  model \eqref{pat-cp} admits a unique semi-trivial equilibrium  $(\bm u^*,\bm 0)$ by Proposition \ref{effp}.

The following estimate about $\bm u^*$ will be useful later.
\begin{lemma}\label{prioriu}
Suppose that  \textbf{H}1 and \textbf{H}2 hold. Let $\bm u^*=(u^*_1,\dots,u^*_n)^T$ be the unique positive solution of \eqref{semitu}.
Then, $0<u^*_1<\dots<u^*_{n}<r$.
\end{lemma}
\begin{proof}
It follows from \eqref{semitu} that
\begin{equation}\label{induct1}
\begin{split}
&(d+q)(u^*_{n-1}-u^*_n)=-u^*_n(r-u^*_n),\\
&(d+q)(u^*_{i-1}-u^*_i)=d(u^*_i-u^*_{i+1})-u^*_i(r-u^*_i),\;\;i=2,\dots,n-1,\\
&qu^*_1=d(u^*_{2}-u^*_1)+u^*_1(r-u^*_1).
\end{split}
\end{equation}
We first claim that $u^*_n<r$. If it is not true, then we see from the first equation of \eqref{induct1} that
$u^*_{n-1}\ge u^*_n\ge r$. By induction,
we obtain from the second equation of \eqref{induct1} that
\begin{equation*}
u^*_1\ge u^*_2\ge \dots\ge u^*_n\ge r.
\end{equation*}
By the third equation of \eqref{induct1},
 $$
 qu^*_1=d(u^*_{2}-u^*_1)+u^*_1(r-u^*_1)\le 0,
 $$
 which contradicts with $u_1^*>0$.
  Therefore,  $u^*_n<r$. Then by virtue of \eqref{induct1}, we obtain that
$u^*_1<u_2^*<\dots<u^*_{n}<r$. This completes the proof.
\end{proof}

By Lemma \ref{prioriu}, we have $\bm{r-u^*}\gg \bm 0$. Therefore, the function $q_{\bm{r-u^*}}^*(d)$ is well-defined for $d\in (0, \infty)$ by Lemma \ref{effp}. Moreover, by Lemma \ref{qstr} (iii), we have $q_{\bm{r-u^*}}^*(d)<q_{\bm{r}}^*(d)$ for all $d>0$.

\begin{proposition}\label{q1pr}
Suppose that  \textbf{H}1 and \textbf{H}2 hold. Then $q_{\bm{r-u^*}}^*(d)$ is strictly increasing for  $d\in (0,\infty)$ with
$$
\ds\lim_{d\to0}q_{\bm{r-u^*}}^*(d)=r-u_1^*>0\ \  \text{and}\ \
\ds\lim_{d\to\infty}q_{\bm{r-u^*}}^*(d)=\ds\sum_{i=1}^n(r-u_i^*)>0.
$$
\end{proposition}
\begin{proof}
 By Lemma \ref{prioriu}, we have $0<u_1^*<\dots<u_n^*<r$. Then the results follow from Lemma \ref{qstr}.
\end{proof}

We partition the first quadrant of the $d-q$ plane into two disjoint subsets:
\begin{equation}\label{S}
\begin{split}
S_1:=&\{(d,q):q,d>0,\;q>q_{\bm{r-u^*}}^*(d)\},\\
S_2:=&\{(d,q):q,d>0,\;q<q_{\bm{r-u^*}}^*(d)\}.
\end{split}
\end{equation}
We have the following result about the local stability of the semi-trivial equilibrium $(\bm u^*,\bm 0)$ of model \eqref{pat-cp}.
\begin{proposition}\label{iv1}
Suppose that  \textbf{H}1 and \textbf{H}2 hold. Then the following statements about the semi-trivial equilibrium $(\bm u^*, \bm 0)$ of \eqref{pat-cp} hold:
\begin{enumerate}
\item [{$\rm (i)$}] If $(d_2, q_2)\in S_1$, then $(\bm u^*,\bm 0)$ is locally asymptotically stable.
\item [{$\rm (ii)$}] If $(d_2, q_2)\in S_2$, then $(\bm u^*,\bm 0)$ is unstable.
\end{enumerate}
\end{proposition}

\begin{proof}
Linearizing \eqref{pat-cp} at $(\bm u^*, \bm 0)$, we can see that its stability  is determined by the sign of $\la_1(d_2, q_2, \bm r-u^*)$: if $\la_1(d_2, q_2, \bm{r-u^*})<0$, it is stable; and if $\la_1(d_2, q_2, \bm{r-u^*})>0$, it is unstable. By Lemma \ref{prop-p}, $\la_1(d, q, \bm{r-u^*})>0$ is decreasing in $q\in (0, \infty)$ for each $d>0$.   Since $\la_1(d, q_{\bm{r-u^*}}^*(d), \bm{r-u^*})=0$, we know $\la_1(d_2,q_2, \bm{r-u^*})<0$ for $(d_2, q_2)\in S_1$, and $\la_1(d_2,q_2, \bm{r-u^*})>0$ for $(d_2,q_2)\in S_2$. Therefore, $(\bm u^*,\bm 0)$ is locally asymptotically stable for $(d_2, q_2)\in S_1$ and is unstable for $(d_2,q_2)\in S_2$.
\end{proof}
To characterize the set $S_1$ and $S_2$ more precisely, we first prove the following property about $q=q_{\bm{r-u^*}}^*(d)$.
\begin{lemma}\label{G}
Suppose that  \textbf{H}1 and \textbf{H}2 hold. Then the two functions $q=q_{\bm{r-u^*}}^*(d)$ and $q=\ds\f{q_1}{d_1}d$ have exactly one intersection point $(d_1, q_1)$ in the first quadrant.
\end{lemma}
\begin{proof}
Since  $\la_1(d_1, q_1, \bm{r-u^*})=0$, we have $q_1=q_{\bm{r-u^*}}^*(d_1)$. Therefore, $d=d_1$ is a root of the equation $q_{\bm{r-u^*}}^*(d)-\ds\f{q_1}{d_1}d=0$. To see this is the only root, we suppose to the contrary that $\bar d_1\neq d_1$ is another root. Without loss of generality, we assume $\bar d_1=\bar\mu d_1$ for some $\bar\mu>1$. So, we have $q_{\bm{r-u^*}}^*(\bar d_1)=\ds\f{q_1}{d_1}\bar d_1=\bar \mu q_1$.
 By Lemma \ref{theorem_quasi}, we have
$$
\ds\f{d}{d\mu} \la_1(\mu d_1,\mu q_1,\bm{r-u^*})<0.
$$
Therefore, we obtain
$$
0=\la_1(d_1,q_1,\bm{r-u^*})>\la_1(\bar\mu d_1,\bar\mu q_1,\bm{r-u^*})=\la_1(\bar d_1, q_{\bm{r-u^*}}^*(\bar d_1),\bm{r-u^*})=0,
$$
which is a contradiction.
\end{proof}

Next we define two subsets of the first quadrant of the $d-q$ plane:
\begin{equation*}
\begin{split}
G_1:=&\{(d,q):0<d\le \frac{d_1}{q_1} q,q\ge q_1,(d,q)\ne(d_1,q_1)\},\\
G_2:=&\{(d,q):d\ge \frac{d_1}{q_1}q,0<q\le q_1,(d,q)\ne(d_1,q_1)\}.
\end{split}
\end{equation*}

By Proposition \ref{q1pr}, function $q_{\bm{r-u^*}}^*(d)$ is strictly increasing in $d$.
By Lemma \ref{G}, we have:
\begin{equation}\label{region1}
   G_1\subset S_1 \ \ \ \text{and}\ \ \ G_2\subset S_2.
\end{equation}
It turns out that we are  able to completely understand the dynamics of  model \eqref{pat-cp} for $(d_2, q_2)\in G_1\cup G_2$. The key ingredient is the following result:

\begin{lemma}\label{iv2}
Suppose that  \textbf{H}1 and \textbf{H}2 hold.  Then if $(d_2,q_2)\in G_1\cup G_2$,  model \eqref{pat-cp} has no positive equilibrium.
\end{lemma}
\begin{proof}
Let $(d_2,q_2)\in G_1\cup G_2$. Suppose to the contrary that model \eqref{pat-cp} admits a positive equilibrium $(\hat{\bm u},\hat{\bm v})$, where
$\hat{\bm u}=(\hat u_1,\dots,\hat u_n)\gg\bm 0$ and  $\hat{\bm v}=(\hat v_1,\dots,\hat v_n)\gg\bm 0$.
Then, we have
\begin{equation*}
\begin{split}
&(d_1+q_1)(\hat u_{n-1}-\hat u_n)=-\hat u_n(r-\hat u_n-\hat v_n),\\
&(d_1+q_1)(\hat u_{i-1}-\hat u_i)=d_1(\hat u_i-\hat u_{i+1})-\hat u_i(r-\hat u_i-\hat v_i),\;\;i=2,\dots,n-1,\\
&q_1\hat u_1=d_1(\hat u_{2}-\hat u_1)+\hat u_1(r-\hat u_1-\hat v_1).
\end{split}
\end{equation*}
and
\begin{equation*}
\begin{split}
&(d_2+q_2)(\hat v_{n-1}-\hat v_n)=-\hat v_n(r-\hat u_n-\hat v_n),\\
&(d_2+q_2)(\hat v_{i-1}-\hat v_i)=d_2(\hat v_i-\hat v_{i+1})-\hat v_i(r-\hat u_i-\hat v_i),\;\;i=2,\dots,n-1,\\
&q_2\hat v_1=d_2(\hat v_{2}-\hat v_1)+\hat v_1(r-\hat u_1-\hat v_1).
\end{split}
\end{equation*}
Then, using similar arguments as in the proof of Lemma \ref{prioriu}, we can show that
$\hat z_1<\hat z_2<\dots<\hat z_n$ for $z=u,v$ and $\hat u_i+\hat v_i<r$ for $1\le i\le n$. Therefore, $\bm{r-\hat u-\hat v}\gg \bm 0$. By Proposition \ref{effp},  function $q=q_{\bm{r-\hat u-\hat v}}^*(d)$ is well-defined for $d\in(0,\infty)$. Moreover, by Lemma \ref{qstr},  it is strictly increasing in $(0, \infty)$.

Noticing that $(\hat{\bm u},\hat{\bm v})$ is a positive equilibrium, we have $\la_1(d_1,q_1,\bm{r-\hat u-\hat v})=\la_1(d_2,q_2,\bm{r-\hat u-\hat v})=0$.
By virtue of similar arguments as in the proof of Lemma \ref{G}, the functions $q=q_{\bm{r-\hat u-\hat v}}^*(d)$ and $q=\ds\f{q_1}{d_1}d$ have exactly one intersection point $(d_1, q_1)$ in the first quadrant  of the $d-q$ plane. It follows that
$$
G_1\subset \{(d,q):q,d>0,\;q>q_{\bm{r-\hat u-\hat v}}^*(d)\}
$$
and
$$
G_2\subset \{(d,q):q,d>0,\;q<q_{\bm{r-\hat u-\hat v}}^*(d)\}.
$$
By Lemma \ref{prop-p}, we have $\la_1(d_2,q_2, \bm{r-\hat u-\hat v})<0$ for $(d_2,p_2)\in G_1$ and  $\la_1(d_2,q_2,\bm{r-\hat u-\hat v})>0$ for $(d_2,p_2)\in G_2$, which contradicts with $\la_1(d_2,q_2,\bm{r-\hat u-\hat v})=0$.
Therefore, model \eqref{pat-cp} has no positive equilibrium if $(d_2,q_2)\in G_1\cup G_2$.
\end{proof}

By virtue of Proposition \ref{iv1}, Lemma \ref{iv2} and the monotone dynamical system theory, we have the following main result about the global dynamics of model \eqref{pat-cp}:
\begin{theorem}\label{theorem_global}
Suppose that  \textbf{H}1 and \textbf{H}2 hold.
Then the following statements hold:
\begin{enumerate}
\item [{$\rm (i)$}] If $(d_2,q_2)\in G_1$, then $(\bm u^*,\bm 0)$ is globally asymptotically stable for  \eqref{pat-cp};
\item [{$\rm (ii)$}] If $(d_2,q_2)\in G_2$, then the semi-trivial
equilibrium $(\bm 0,\bm v^*)$ exists and is globally asymptotically stable for  \eqref{pat-cp}.
\end{enumerate}
\end{theorem}
\begin{proof}
{$\rm (i)$} Suppose $(d_2,q_2)\in G_1$. We claim that semi-trivial equilibrium $(\bm 0,\bm v^*)$ is either unstable or does not exist. Indeed, if $q_2<q_{\bm r}^*(d_2)$ then  $(\bm 0,\bm v^*)$ exists. Since the nonlinear terms of the model are symmetric and $(\bm u^*,\bm 0)$ is unstable when $(d_2,q_2)\in G_2$,  $(\bm 0,\bm v^*)$ is unstable when $(d_2,q_2)\in G_1$. If  $q_2\ge q_{\bm r}^*(d_2)$, then $(\bm 0,\bm v^*)$ does not exist. Since model \eqref{pat-cp} has no positive equilibrium for $(d_2,q_2)\in G_1$, by the monotone dynamical system theory \cite{hess,hsu1996competitive, LAM2016Munther,smith2008monotone},  $(\bm u^*,\bm 0)$ is globally asymptotically stable.

{$\rm (ii)$} By Lemma \ref{qstr}, we have $q_{\bm{r-u^*}}^*(d)<q_{\bm{r}}^*(d)$ for all $d>0$. So if  $(d_2,q_2)\in G_2\subset \{(d,q):q,d>0,\;q<q_{\bm{r}}^*(d)\}$, semi-trivial equilibrium $(\bm 0,\bm v^*)$ exists. Since $(\bm u^*,\bm 0)$ is unstable and model \eqref{pat-cp} has no positive equilibrium for $(d_2,q_2)\in G_2$, the desired result follows from the monotone dynamical system theory  \cite{hess,hsu1996competitive, LAM2016Munther,smith2008monotone}.
\end{proof}

\begin{remark}
Theorem \ref{theorem_global} is illustrated in Figure \ref{FigA}. We are able to completely understand the global dynamics of model \eqref{pat-cp} in the colored regions ($G_1$ and $G_2$), in which the competitive exclusion happens.
\end{remark}

We have the following observations from Theorem \ref{theorem_global}:
\begin{corollary}\label{corollary_global}
Suppose that  \textbf{H}1 and \textbf{H}2 hold.
Then the following statements hold:
\begin{enumerate}
\item [{$\rm (i)$}] Fix $q_2=q_1$. If $d_2<d_1$,  $(\bm u^*,\bm 0)$ is globally asymptotically stable for  \eqref{pat-cp}; and if $d_2>d_1$, $(\bm 0,\bm v^*)$ is globally asymptotically stable;
\item [{$\rm (ii)$}] Fix $d_2=d_1$. If $q_2>q_1$,  $(\bm u^*,\bm 0)$ is globally asymptotically stable for \eqref{pat-cp}; and if $q_2<q_1$, $(\bm 0,\bm v^*)$ is globally asymptotically stable.
\end{enumerate}
\end{corollary}

\begin{remark}
By Corollary \ref{corollary_global}, the species with a larger diffusion rate or a smaller advection rate can invade and  replace the resident species.
Moreover, Corollary \ref{corollary_global} {$\rm (i)$}  resolves a conjecture in \cite{lou2019ideal}, which was originally proposed for a two-patch model.
\end{remark}

\begin{figure}
 \centering
  \includegraphics[width=0.6\textwidth]{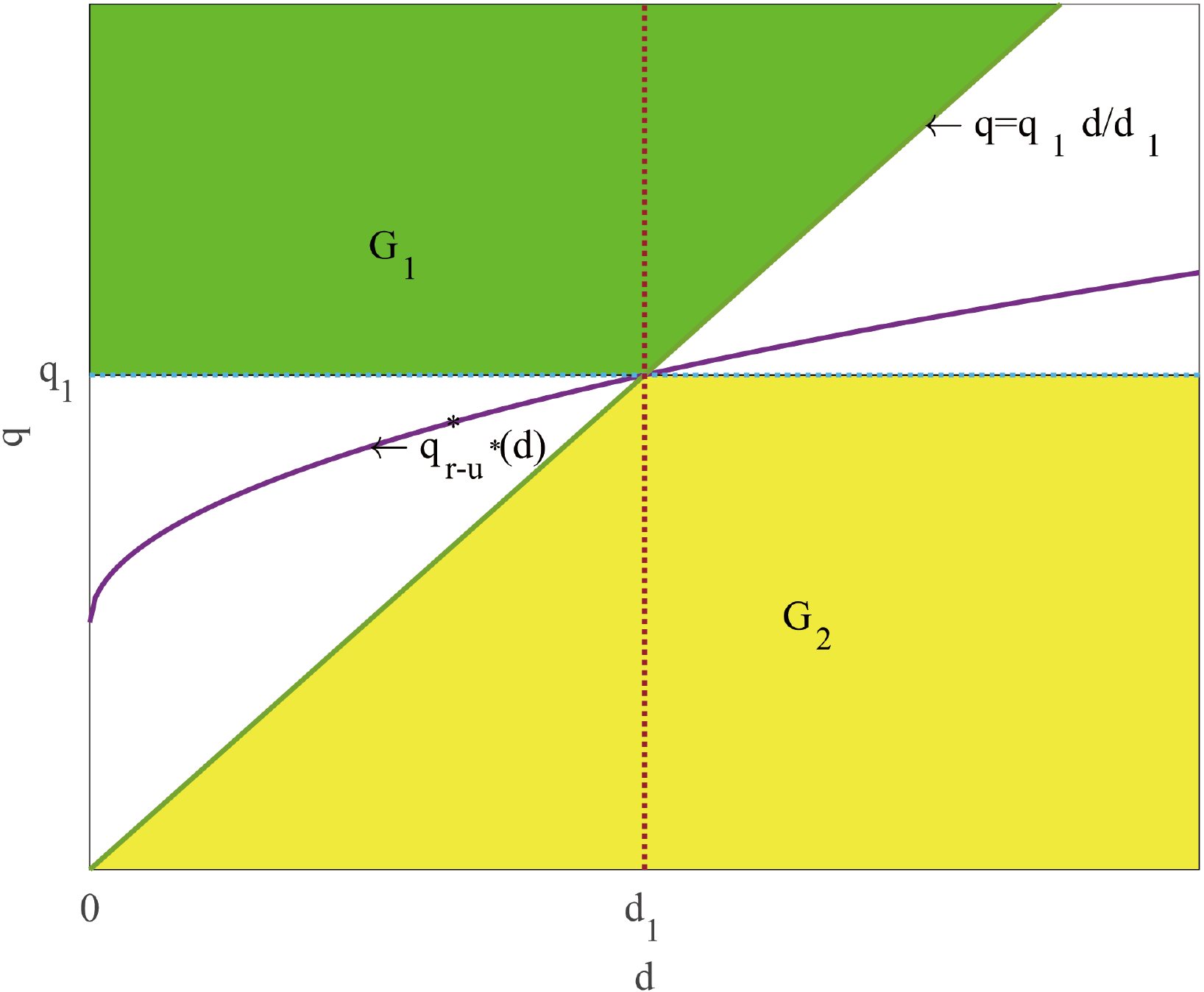}
 \caption{Illustration of Proposition \ref{iv1} and Theorem \ref{theorem_global} for model \eqref{pat-cp}-case (a). Here, $d_1$ and $q_1$ are fixed and satisfy \textbf{H}2 such that  $(\bm u^*,\bm 0)$ exists. If $(d_2, q_2)$ is above the curve $q=q_{\bm{r-u^*}}^*(d)$, then  $(\bm u^*,\bm 0)$ is stable; if $(d_2, q_2)$ is under the curve, then $(\bm u^*,\bm 0)$ is unstable.
 If $(d_2, q_2)\in G_1$,   $(\bm u^*,\bm 0)$ is globally asymptotically stable; and if $(d_2, q_2)\in G_2$, $(\bm 0,\bm v^*)$ exists and is globally asymptotically stable. }
 \label{FigA}
 \end{figure}

To study the dynamics of model \eqref{pat-cp} with $(d_2,q_2)$ in the region other than $G_1,G_2$ in Fig. \ref{FigA}, we first  show that there exists a curve which determines the stability of the semi-trivial equilibrium $(\bm 0, \bm v^*)$. We recall that $(\bm 0, \bm v^*)$ exists if and only if $0<q_2<q_{\bm r}^*(d_2)$.
\begin{proposition}\label{q2pr}
Suppose that  \textbf{H}1 and \textbf{H}2 hold. Then the following statements hold:
\begin{enumerate}
    \item [{$\rm (i)$}] For any $d_2>0$, $\la_1(d_1,q_1,\bm{r-v^*}(d_2,q_2))$ is strictly increasing in $q_2$ for $q_2\in(0, q_{\bm r}^*(d_2))$;
    \item [{$\rm (ii)$}] For any $d_2>0$, there exists a unique $q_{2}^*(d_2)\in (0, q_{\bm r}^*(d_2))$ such that
    \begin{equation}\label{q20}
      \la_1(d_1,q_1,\bm{r-v^*}(d_2,q^*_{2}(d_2)))=0;
    \end{equation}
Moreover, $q_2^*(d_1)=q_1$, and
\begin{equation}\label{llin}
\begin{split}
&q_1<q_2^*(d_2)<\min\left\{\ds\f{q_1}{d_1}d_2,q_{\bm r}^*(d_2)\right\} \;\;\text{for}\;\; d_2>d_1,\\
&\ds\f{q_1}{d_1}d_2<q_2^*(d_2)<\min\left\{q_1,q_{\bm r}^*(d_2)\right\}\;\;\text{for}\;\; d_2<d_1.
\end{split}
\end{equation}
\item [{$\rm (iii)$}] Semi-trivial equilibrium $(\bm 0,{\bm v}^*(d_2,q_2))$ is stable if
$0<q_2<q_2^*(d_2)$ and unstable if $q_2^*(d_2)<q_2<q_{\bm r}^*(d_2)$.
\end{enumerate}
\end{proposition}
\begin{proof}
{$\rm (i)$} Clearly, $v^*(d_2,q_2)$ satisfies
\begin{equation}\label{semitv2}
\ds\sum_{j=1}^{n}(d_2D_{ij}+q_2Q_{ij})v_j+v_i(r-v_i)=0, \;\;i=1,\dots,n,
\end{equation}
Differentiating \eqref{semitv2} with respect to $q_2$ yields
\begin{equation}\label{derivsemitv2}
-\ds\sum_{j=1}^{n}(d_2D_{ij}+q_2Q_{ij})v'_j-(r-2v_i)v_i'=\sum_{j=1}^n Q_{ij}v_j , \;\;i=1,\dots,n,
\end{equation}
Denote $L=d_2D+q_2Q+\text{diag}(r-2v_i)$, and we have
$s(L)<s(d_2D+q_2Q+\text{diag}(r-v_i))=0$. So $-L$ is a non-singular $M$-matrix and $-L^{-1}$ is a positive matrix \cite{berman1994nonnegative}. By
Lemma \ref{prioriu}, we have $v_1<\dots<v_n$. Therefore, the right hand side of \eqref{derivsemitv2} is negative. This implies that $v_i'<0$ for $i=1,..., n$ and each entry of $v^*(d_2,q_2)$ is strictly decreasing in $q_2$ for $q\in(0,q_{\bm r}^*(d_2))$. Then, it follows from Lemma \ref{prop-p} that {$\rm (i)$} holds.

{$\rm (ii)$}  We only consider the case  $0< d_2<d_1$, since the other case $d_2>d_1$ is similar. It follows from Theorem
\ref{theorem_global} that $(\bm 0,\bm v^*)$ is stable for $q_2\le \ds\f{q_1}{d_1}d_2$ and unstable for $q_2\ge q_1$.
Hence, we have
$\la_1(d_1,q_1,\bm{r-v^*}(d_2,q_{2}))< 0$ for $q_2\le \ds\f{q_1}{d_1}d_2$ and $\la_1(d_1,q_1,\bm{r-v^*}(d_2,q_{2}))> 0$ for $q_2\ge q_1$.
Note that
$$
\ds\lim_{q_2\to q^*_2(d_2)^-}\la_1(d_1,q_1,\bm{r-v^*}(d_2,q_{2}))=\la_1(d_1,q_1,\bm r)>0.
$$
Since $\la_1(d_1,q_1,\bm{r-v^*}(d_2,q_2))$ is strictly increasing in $q_2$, there exists unique $q_2^*(d_2)$ satisfying \eqref{q20}-\eqref{llin}.
Clearly, $q_2^*(d_1)=q_1$, and {$\rm (iii)$} follows from {$\rm (i)$}-{$\rm (ii)$}. This completes the proof.
\end{proof}

By Propositions \ref{q1pr} and \ref{q2pr} and the monotone dynamical system theory, we have the following result.
\begin{theorem}\label{theorem_bc}
Suppose that  \textbf{H}1 and \textbf{H}2 hold. Let $q_2^*(d_2)$ be defined in Proposition \ref{q2pr}. Then the following statements hold:
\begin{enumerate}
\item[{$\rm (i)$}] If $q_2^*(d_2)<q^*_{\bm {r-u^*}}(d_2)$, then for any $q_2\in \left(q_2^*(d_2),q^*_{\bm {r-u^*}}(d_2)\right)$ both $(\bm u^*,\bm 0)$ and $(\bm 0, \bm v^*)$ are  unstable, and system \eqref{pat-cp} admits a stable positive equilibrium;
\item [{$\rm (ii)$}] If $q_2^*(d_2)>q^*_{\bm {r-u^*}}(d_2)$ then for any $q_2\in \left(q^*_{\bm {r-u^*}}(d_2), q_2^*(d_2)\right)$ both $(\bm u^*,\bm 0)$ and $(\bm 0, \bm v^*)$ are stable, and system \eqref{pat-cp} admits an unstable positive equilibrium.
\end{enumerate}
\end{theorem}
\begin{remark}
Since $\la_1(d_1,q_1,\bm{r-v^*}(d_2,q_2))$ is strictly increasing in $q_2$, the sign of $q_2^*(d_2)-q^*_{\bm {r-u^*}}(d_2)$ is determined by the sign of
\begin{equation}\label{la1la}
\la^*_1(d_2):=\la_1\left(d_1,q_1,\bm{r-v^*}(d_2,q^*_{\bm {r-u^*}}(d_2))\right).
\end{equation}
If $\la^*_1(d_2)>(<)0$, then $q_2^*(d_2)<(>)q^*_{\bm {r-u^*}}(d_2)$.
\end{remark}

\begin{remark}
Some of our results were known for the corresponding reaction-diffusion models, a similar result of Lemma \ref{prioriu} was obtained in \cite{vasilyeva2011population}, a similar result of Corollary \ref{corollary_global} (\rm i) was presented in \cite{lou2014evolution}, and similar results of Lemma \ref{iv2} and Theorem \ref{theorem_global} were proved in \cite{zhou2018global} using a more sophisticated method. We used the monotonicity property of $\bm u^*$ in $q$ in the proof of Proposition \ref{q2pr}, which was also known for reaction-diffusion models \cite{vasilyeva2011population}. Similar results of Proposition \ref{q2pr} and Theorem \ref{theorem_bc} were proved in \cite{LouNie2018}
with respect to some other parameters.
To our best knowledge, similar results of Propositions \ref{q1pr}, \ref{iv1} and Lemma \ref{G} were not proved for reaction-diffusion models.
\end{remark}



\subsection{Invasion analysis for case (b)}

In this subsection, we suppose that \textbf{H}1$^*$ holds. By Lemma \ref{theorem_quasi}, there exists a unique $d^*>0$ such that $\la_1(d^*, 0, \bm r)=0$. By  Theorem \ref{qb},  for any $d_1\in (0, d^*)$ there exists  $q_{\bm r}^*(d_1)>0$ such that $\la_1(d_1, q_1, \bm r)>0$ for $q_1<q_{\bm r}^*(d_1)$ and $\la_1(d_1, q_1, \bm r)<0$ for $q_1>q_{\bm r}^*(d_1)$.
We suppose that species \textbf{u} is established,  i.e.,
\begin{enumerate}
\item[\textbf{H}2$^*$.] $0<d_1<d^*$ and $q_1<q_{\bm r}^*(d_1)$.
\end{enumerate}
Under assumptions \textbf{H}1$^*$ and \textbf{H}2$^*$, model \eqref{pat-cp} admits a unique species $\bm u-$only semi-trivial equilibrium  $(\bm u^*,\bm 0)$. 

We first prove an estimate of $\bm{u^*}$, which will be useful later.
\begin{lemma}\label{prioriu2}
Suppose that  \textbf{H}1$^*$ and \textbf{H}2$^*$ hold. Let $\bm u^*=(u^*_1,\dots, u^*_n)^T$ be the unique positive solution of \eqref{semitu}.
Then, we have $\bm 0\ll\bm{u^*}\ll r$.
\end{lemma}
\begin{proof}
Let $\hat {\bm u}=\bm r$. It is easy to check that $\hat {\bm u}$ is an upper solution of \eqref{semitu}. By the method of upper and lower solutions and the uniqueness of the positive solution of  \eqref{semitu}, we have $\bm 0\ll\bm {u^*}\ll\bm r$.
\end{proof}

By Lemma  \ref{theorem_quasi}, there exists a unique $d^{**}>0$ such that $\la_1(d^{**}, 0, \bm{r- u^*})=0$. By  Proposition \ref{qb},  for any $d_2\in (0, d^{**})$ there exists  $q_{\bm{r- u^*}}^*(d_2)>0$ such that $\la_1(d_2, q_2, \bm{r- u^*})>0$ for $q_2<q_{\bm{r- u^*}}^*(d_2)$ and $\la_1(d_2, q_2, \bm{r- u^*})<0$ for $q_2>q_{\bm{r- u^*}}^*(d_2)$. Moreover, if $d_2\ge d^{**}$, then $\la_1(d_2, q_2, \bm{r- u^*})\le 0$ for all $q_2>0$.
This suggests us to define the following parameter sets:
\begin{equation*}
\begin{split}
    S_1^*:=&\{(d,q): \;q>q_{\bm{r-u^*}}^*(d), \; 0<d<d^{**}\}\cup \{(d,q):q>0, \; d\ge d^{**}\},\\
    S_2^*:=&\{(d,q): 0<q<q_{\bm{r-u^*}}^*(d), \; 0<d<d^{**}\}.
\end{split}
\end{equation*}

We have the following result about the local stability of  semi-trivial equilibrium $(\bm u^*,\bm 0)$ of model \eqref{pat-cp}, and its proof is omitted as it is similar to  Proposition \ref{iv1}.
\begin{proposition}\label{theorem_iv2}
Suppose that  \textbf{H}1$^*$ and \textbf{H}2$^*$ hold. The the following statements about semi-trivial equilibrium $(\bm u^*, \bm 0)$ of \eqref{pat-cp} hold:
\begin{enumerate}
\item [{$\rm (i)$}] If $(d_2, q_2)\in S_1^*$, then $(\bm u^*,\bm 0)$ is locally asymptotically stable;
\item [{$\rm (ii)$}] If $(d_2, q_2)\in S_2^*$, then $(\bm u^*,\bm 0)$ is unstable.
\end{enumerate}
\end{proposition}

We also have the following property about the function $q=q_{\bm{r-u^*}}^*(d)$, and its proof is exactly the same as that of Lemma \ref{G} so we omit it.
\begin{lemma}\label{G1}
Suppose that  \textbf{H}1$^*$ and \textbf{H}2$^*$ hold. Then the two functions $q=q_{\bm{r-u^*}}^*(d)$ and $q=\ds\f{q_1}{d_1}d$ have exactly one intersection point $(d_1, q_1)$ in the first quadrant.
\end{lemma}

Again we define parameter sets:
\begin{equation*}
\begin{split}
G_1^*:=&\{(d,q):d_1<d\le \frac{d_1}{q_1} q, (d,q)\ne(d_1,q_1)\},\\
G_2^*:=&\{(d,q): \frac{d_1}{q_1}q\le d\le d_1, q>0,(d,q)\ne(d_1,q_1)\}.
\end{split}
\end{equation*}
Since  the two functions $q=q_{\bm{r-u^*}}^*(d)$ and $q=\ds\f{q_1}{d_1}d$ have exactly one intersection point $(d_1, q_1)$ in the first quadrant, we have that $G_1^*\subset S_1^*$ and $G_2^*\subset S_2^*$.
The following result is similar to Lemma \ref{iv2}.

\begin{lemma}\label{iv23}
Suppose that  \textbf{H}1$^*$ and \textbf{H}2$^*$ hold.  Then if $(d_2,q_2)\in G_1^*\cup G_2^*$,  model \eqref{pat-cp} has no positive equilibrium.
\end{lemma}
\begin{proof}
Let $(d_2,q_2)\in G_1^*\cup G_2^*$. Suppose to the contrary that model \eqref{pat-cp} admits a positive equilibrium $(\hat{\bm u},\hat{\bm v})$, where
$\hat{\bm u}=(\hat u_1,\dots,\hat u_n)\gg\bm 0$ and  $\hat{\bm v}=(\hat v_1,\dots,\hat v_n)\gg\bm 0$.
Then, we have
\begin{equation}\label{ind-uv11}
\begin{split}
&(d_1+q_1)(\hat u_{n-1}-\hat u_n)=d_1\hat u_n-\hat u_n(r-\hat u_n-\hat v_n),\\
&(d_1+q_1)(\hat u_{i-1}-\hat u_i)=d_1(\hat u_i-\hat u_{i+1})-\hat u_i(r-\hat u_i-\hat v_i),\;\;i=2,\dots,n-1,\\
&q_1\hat u_1=d_1(\hat u_{2}-\hat u_1)+\hat u_1(r-\hat u_1-\hat v_1).
\end{split}
\end{equation}
and
\begin{equation}\label{ind-uv12}
\begin{split}
&(d_2+q_2)(\hat v_{n-1}-\hat v_n)=d_2\hat v_n-\hat v_n(r-\hat u_n-\hat v_n),\\
&(d_2+q_2)(\hat v_{i-1}-\hat v_i)=d_2(\hat v_i-\hat v_{i+1})-\hat v_i(r-\hat u_i-\hat v_i),\;\;i=2,\dots,n-1,\\
&q_2\hat v_1=d_2(\hat v_{2}-\hat v_1)+\hat v_1(r-\hat u_1-\hat v_1).
\end{split}
\end{equation}
We claim  $\ds\max_{1\le i\le n}\{r-\hat u_i-\hat v_i\}>0$. Suppose to the contrary that $r-\hat u_i-\hat v_i\le 0$ for all $1\le i\le n$. Then, by the first two equations in both \eqref{ind-uv11} and \eqref{ind-uv12}, we obtain that
$\hat z_1>\hat z_2>\dots>\hat z_n$ for $z=u,v$. Then, by the third equation in \eqref{ind-uv11}, we get $q\hat u_1=d(\hat u_{2}-\hat u_1)+\hat u_1(r-\hat u_1-\hat v_1)<0$, which is a contradiction. Therefore,  $\ds\max_{1\le i\le n}\{r-\hat u_i-\hat v_i\}>0$. By Proposition \ref{qb},  function $q=q_{\bm{r-\hat u-\hat v}}^*(d)$ is well-defined. Similar to Lemma \ref{G1},  functions $q=q_{\bm{r-\hat u-\hat v}}^*(d)$ and $q=\ds\f{q_1}{d_1}d$ have exactly one intersection point $(d_1, q_1)$ in the first quadrant of the $d-q$ plane. The rest of the proof is similar to that of Lemma \ref{iv2}, so we omit it here.
\end{proof}

By Proposition \ref{theorem_iv2}, Lemma \ref{iv23} and the monotone dynamical system theory, we have the following result about the global dynamics of model \eqref{pat-cp}. We omit the proof as it is similar to Theorem \ref{theorem_global}.
\begin{theorem}\label{theorem_global2}
Suppose that  \textbf{H}1$^*$ and \textbf{H}2$^*$ hold.
Then the following statements hold:
\begin{enumerate}
\item [{$\rm (i)$}] If $(d_2,q_2)\in G_1^*$,   semi-trivial
equilibrium $(\bm u^*,\bm 0)$ is globally asymptotically stable for model \eqref{pat-cp};
\item  [{$\rm (ii)$}] If $(d_2,q_2)\in G_2^*$,   semi-trivial
equilibrium $(\bm 0,\bm v^*)$ exists and is globally  asymptotically stable for model \eqref{pat-cp}.
\end{enumerate}
\end{theorem}

\begin{remark}
Theorem  \ref{theorem_global2} is illustrated in Figure \ref{FigB}. We are able to completely understand the global dynamics of model \eqref{pat-cp} in the colored regions, in which  competitive exclusion occurs.
\end{remark}

\begin{figure}[htbp]
 \centering
  \includegraphics[width=0.6\textwidth]{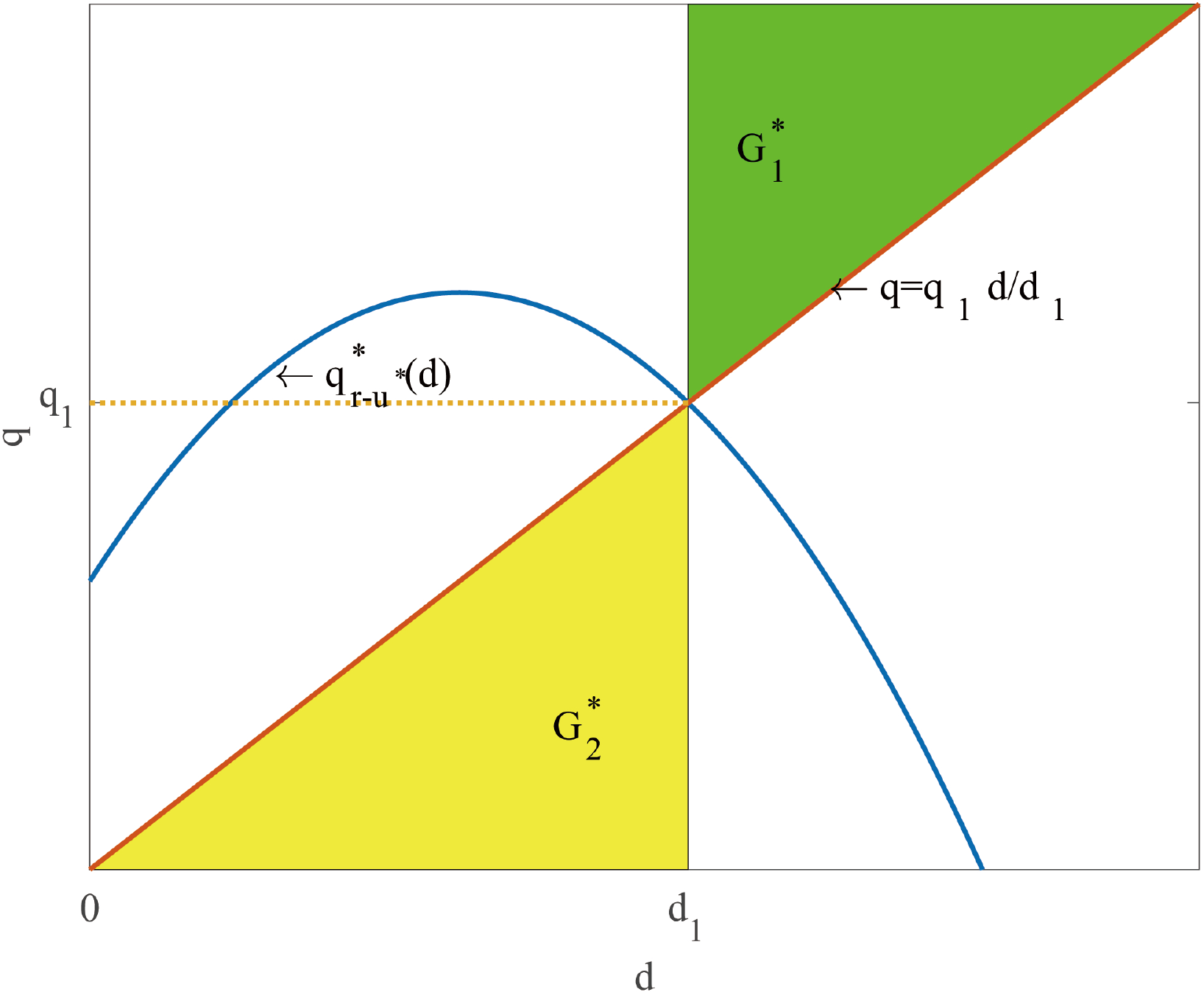}
 \caption{Illustration of Proposition \ref{theorem_iv2} and Theorem \ref{theorem_global2} for model \eqref{pat-cp}-case (b). Here, $d_1$ and $q_1$ are fixed and satisfy \textbf{H}2$^*$ such that  $(\bm u^*,\bm 0)$ exists. If $(d_2, q_2)$ is above the curve $q=q_{\bm{r-u^*}}^*(d)$, then  $(\bm u^*,\bm 0)$ is stable; if $(d_2, q_2)$ is under the curve, then $(\bm u^*,\bm 0)$ is unstable. If $(d_2, q_2)\in G_1^*$,   $(\bm u^*,\bm 0)$ is globally asymptotically stable; and if $(d_2, q_2)\in G_2^*$, $(\bm 0,\bm v^*)$ is globally asymptotically stable.
 }
 \label{FigB}
 \end{figure}

We have the following observations from Theorem \ref{theorem_global2}:
\begin{corollary}\label{corollary_global2}
Suppose that  \textbf{H}1$^*$ and \textbf{H}2$^*$ hold. Fix $d_2=d_1$. If $q_2>q_1$,  $(\bm u^*,\bm 0)$ is globally asymptotically stable for  \eqref{pat-cp};  if $q_2<q_1$, $(\bm 0,\bm v^*)$ is globally asymptotically stable.
\end{corollary}

By Corollary \ref{corollary_global2},  the species with a smaller advection rate can invade and  replace the resident species in this case. However, whether a smaller or larger diffusion rate is favored seems to be more complicated in case (b).  If we treat $\bm v$ as the mutating species such that $d_2$ is close to $d_1$ and $q_1=q_2$, then $\left[q_{\bm{r-u^*}}^*(d)\right]'|_{d=d_1}>0$ means that $\bm v$ can invade if and only if $d_2>d_1$; $\left[q_{\bm{r-u^*}}^*(d)\right]'|_{d=d_1}<0$ means that $\bm v$ can invade if and only if $d_2<d_1$. We will show that the sign $\left[q_{\bm{r-u^*}}^*(d)\right]'|_{d=d_1}$ is not definite.

 Define $\mathcal{S}:=\{(d,q):0<d<d^*,\;0<q<q^*_{\bm r}(d)\}$, and recall that $(\bm u^*, \bm 0)$ exists if and only if $(d_1, q_1)\in \mathcal{S}$.
 Let $\mathcal{S}':=\{(d,q): 0<q<r, \; q=q^*_{\bm r}(d)\}$. By Lemma \ref{qstr2}, $\mathcal{S}'$ is non-empty and it is a curve $q=q^*_{\bm r}(d)$ connecting to  $(d^*, 0)$.

\begin{proposition}\label{q1pr2}
Suppose that  \textbf{H}1* and \textbf{H}2* holds. Then, $\left[q_{\bm{r-u^*}}^*(d)\right]'|_{d=d_1}$ changes sign for $(d_1,q_1)\in \mathcal{S}$.
Moreover,
\begin{enumerate}
    \item [{$\rm (i)$}] $\ds\lim_{(d_1,q_1)\to (d_0,0)}\left[q_{\bm{r-u^*}}^*(d)\right]'|_{d=d_1}<0$ for any $d_0\in(0,d^*)$;
    \item [{$\rm (ii)$}] $\ds\lim_{(d_1,q_1)\to (0,q_0)}\left[q_{\bm{r-u^*}}^*(d)\right]'|_{d=d_1}>0$ for any $q_0\in(0,r)$;
     \item [{$\rm (iii)$}] $\ds\lim_{(d_1,q_1)\to (d_0,q_0)}\left[q_{\bm{r-u^*}}^*(d)\right]'|_{d=d_1}<0$ for any  $(d_0, q_0)\in \mathcal{S}'$.
\end{enumerate}
\end{proposition}
\begin{proof}
For simplicity, we denote $q^*_{\bm{r- u^*}}(d)$ by  $q_1^*(d)$.
Let $\bm \phi=(\phi_1,\phi_2,\dots,\phi_n)^T\gg \bm 0$ be the  eigenvector corresponding to the  eigenvalue $\la_1(d,q_1^*(d), \bm{r-u^*})=0$ with $\sum_{i=1}^n\phi_i=1$. Then, we have
\begin{equation}\label{q1eigv}
d\sum_{j=1}^n D_{ij}\phi_j+q_1^*(d)\sum_{j=1}^n Q_{ij}\phi_j+(r-u_i^*)\phi_i=0,\ \ i=1,...,n.
\end{equation}
 Similar to the proof of Lemma \ref{qstr2} {$\rm (iii)$}, we can show
\begin{equation}\label{deriv3}
   \left[q_{1}^*(d)\right]'=-\ds\f{\ds\sum_{i=1}^n\sum_{j=1}^n\beta_iD_{ij}\phi_i\phi_j}{\ds\sum_{i=1}^n\sum_{j=1}^n\beta_iQ_{ij}\phi_i\phi_j},
\end{equation}
where
 \begin{equation*}
(\beta_1, \beta_2, \beta_3,\dots, \beta_n)=\left(1,\ds\frac{d}{d+q_1^*(d)},\left(\ds\frac{d}{d+q_1^*(d)}\right)^2,\dots, \left(\ds\frac{d}{d+q_1^*(d)}\right)^{n-1}\right).
\end{equation*}
Since $\bm {u^*}$ is an eigenvector corresponding to $\la_1(d_1, q_1, \bm{r- u^*})(=0)$, $q_1^*(d_1)=q_1$ and $\bm \phi$ is a multiple of $\bm u^*$ when $d=d_1$.


It is easy to see that
\begin{equation}\label{deriv4}
\begin{split}
   & \lim_{(d_1,q_1)\to (d_0,0)} (\beta_1,\dots,\beta_n)=(1,\dots,1),\\
    &\lim_{(d_1,q_1)\to (0,q_0)} (\beta_1,\dots,\beta_n)=(1,0,\dots,0),
    \end{split}
\end{equation}
and
\begin{equation*}
  \ds\lim_{(d_1,q_1)\to (d_0,0)}\bm{u^*}=\bm{\tilde u^*} ,\;\;\text{and}\;\; \ds\lim_{(d_1,q_1)\to (0,q_0)}\bm {u^*} =\bm{\hat u^*},
\end{equation*}
where $\bm{\tilde u^*}$ and $\bm{\hat u^*}$ satisfy \eqref{semitu} with $(d_1,q_1)=(d_0,0)$ and $(d_1,q_1)=(0,q_0)$, respectively.
A direct computation implies that
\begin{equation*}
\tilde u_1^*>\dots>\tilde u_n^*>0,\;\;\text{and}\;\;
0<r-q_0=\hat u_1^*<\dots<\hat u_n^*.
\end{equation*}
This combined with \eqref{deriv3}-\eqref{deriv4} implies {$\rm (i)$}-{$\rm (ii)$}.

Finally, we prove {$\rm (iii)$}. We claim that $\lim_{(d_1,q_1)\to (d_0, q_0)}\bm{u^*}=\bm 0$. To see it, suppose to the contrary that, up to a subsequence, $\bm{u^*}\to \bm{{\check u}^*}\neq \bm 0$ as $(d_1,q_1)\to (d_0,q_0)$. Then, we have $\bm 0<\bm{{\check u}^*}\le\bm r$ and $\lambda_1(d_0, q_0, \bm r-\bm{{\check u}^*})=\ds\lim_{(d_1,q_1)\to (d_0, q_0)} \lambda_1(d_1, q_1, \bm{r- u^*})=0$. Since $(d_0, q_0)\in \mathcal{S}'$, we have $\lambda_1(d_0, q_0, \bm r)=0$. This implies $\lambda_1(d_0, q_0, \bm r-\bm{{\check u}^*})<\lambda_1(d_0, q_0, \bm r)=0$, which is a contradiction.

Letting $d=d_1$ in \eqref{q1eigv}, using $q^*_1(d_1)=q_1$ and $\bm{u^*}\to\bm 0$, we obtain  $\bm\phi\to \check {\bm\phi}$ as $(d_1,q_1)\to (d_0, q_0)$, where $\check {\bm\phi}$ satisfies
\begin{equation}\label{q1eigvlim}
d_0\sum_{j=1}^n D_{ij}\check\phi_j+q_0\sum_{j=1}^n Q_{ij}\check\phi_j+r\check\phi_i=0, \ \ i=1,...,n.
\end{equation}
We may rewrite \eqref{q1eigvlim} as
\begin{equation}\label{phidd22}
\begin{split}
&d_0(\check\phi_1-\check\phi_2)=(r-q_0)\check\phi_1,\\
&d_0(\check\phi_{i}-\check\phi_{i+1})=(d_0+q_0)(\check\phi_{i-1}-\check\phi_i)+r\check \phi_i,\;\;i=2,\dots,n-1,\\
&d_0(\check\phi_{n-1}-2\check\phi_n)+q_0(\check\phi_{n-1}-\check\phi_n)+r\check\phi_n=0.
\end{split}
\end{equation}
This combined with $q_0<r$ implies $\check \phi_1>\dots>\check\phi_n$. Therefore,
\begin{equation}\label{phidd33}
\sum_{i=1}^n\sum_{j=1}^n\check\beta_iD_{ij}\check\phi_i\check\phi_j = \ds\sum_{i=1}^{n-1}\check\beta_i\left(\check\phi_{i+1}-\check\phi_i\right)\left[\check\phi_{i}-\left(\ds\frac{d_0}{d_0+q_0}\right)\check\phi_{i+1}\right]-\check\beta_n\check\phi_n^2<0,
\end{equation}
where $(\check \beta_1,\dots,\check\beta_n)=(1,\dots,(\frac{d_0}{d_0+q_0})^{n-1})$. Similar to \eqref{sp7}, we can show
\begin{equation}\label{phidd44}
\sum_{i=1}^n\sum_{j=1}^n\check\beta_iQ_{ij}\check\phi_i\check\phi_j<0.
\end{equation}
Evaluating \eqref{deriv3} at $d=d_1$, taking $(d_1, q_1)\to (d_0, q_0)$, and using \eqref{phidd33}-\eqref{phidd44}, we obtain {$\rm (iii)$}.
\end{proof}

\begin{remark}
By Proposition \ref{q1pr2} and Lemma \ref{qstr2}, for each $q_1\in (0, r)$, there exists at least one $\bar d_1>0$ such that $\left[q_{\bm{r-u^*}}^*(d)\right]'|_{d=\bar d_1}=0$. Moreover, the sign of $\left[q_{\bm{r-u^*}}^*(d)\right]'|_{d=d_1}$ switches from positive to negative at $d_1=\bar d_1$.  This suggests that the evolutionary singular strategy $d_1=\bar d_1$ may be a locally convergent stable strategy.
\end{remark}

\begin{remark}
For corresponding reaction-diffusion models, a similar result of Theorem \ref{theorem_global2} was proved in a very recent paper \cite{yan2022competition}. To our best knowledge, similar results of Propositions \ref{theorem_iv2} and \ref{q1pr2} were not proved for reaction-diffusion models.
\end{remark}



 \FloatBarrier

 \section{Simulations and discussions}

In this section, we perform some numerical simulations when model \eqref{pat-cp} is coupled with free-flow (case (a)) or hostile (case (b)) boundary conditions.
 We consider four patches, i.e. $n=4$ and set $r=2$.

 \subsection{Simulations for case (a)}
We choose $d_1=1, q_1=0.5$ such that $(\bm u^*, \bm 0)$ exists for case (a).  We first plot the curve $q=q_{\bm{r-u^*}}^*(d)$, which divide the first quadrant into two subregions.  In Fig. \ref{figqa}, we see that the curve  $q=q_{\bm{r-u^*}}^*(d)$ is monotone increasing, which is expected due to Lemma \ref{qstr}.  By Propositions \ref{iv1}, $E_1=(\bm u^*, \bm 0)$ is stable if $(d_2, q_2)$ is above the curve  $q=q_{\bm{r-u^*}}^*(d)$ and unstable if it is below the curve.
 \begin{figure}[h]
\centering
\includegraphics[height=3in]{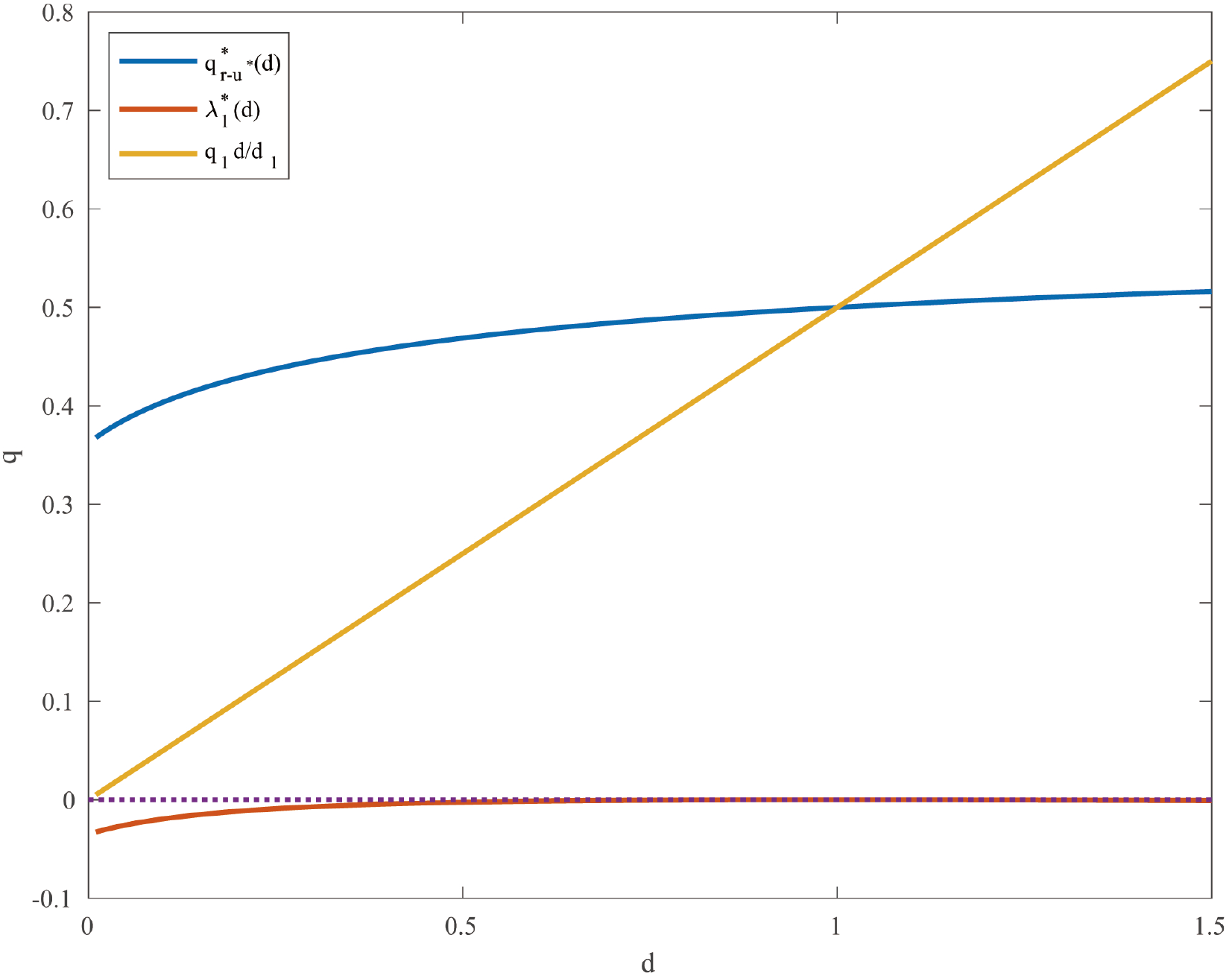}
\vskip -15pt
\caption{Curve $q=q_{\bm{r-u^*}}^*(d)$  with $n=4$, $r=2$, $d_1=1$, $q_1=0.5$ for case (a). The sign of the curve $q=\la^*_1(d)$ determines the stability of $E_2$ when $d_2=d$ and $q_2=q_{\bm{r-u^*}}^*(d)$.}
\label{figqa}
\vskip -8pt
\end{figure}

 We further consider the stability of $E_2=(\bm 0, \bm v^*)$ when $(d_2, q_2)=(d, q_{\bm{r-u^*}}^*(d))$, which is determined by the sign of the principal eigenvalue
 $$
\la^*_1(d_2)=\la_1\left(d_1,q_1,\bm{r-v^*}(d_2,q^*_{\bm {r-u^*}}(d_2))\right).
 $$
 In Fig. \ref{figqa}, the curve $\la^*_1(d)$ seems to be below zero. Therefore, if we choose $(d_2, q_2)$ sufficiently close  to but above the curve $\la^*_1(d)$, then both
 $(\bm u^*, \bm 0)$ and  $(\bm 0, \bm v^*)$ are stable and we have bi-stability.
To confirm this, we choose $(d_2, q_2)=(0.08, 0.44)$, which is slightly above the curve $q=q_{\bm{r-u^*}}^*(d)$. In Fig. \ref{fig_bi}, we plot the solutions of \eqref{pat-cp} with  different initial data. If the initial data  is $\bm u(0)=(0.1, 0.1, 0.1, 0.1)$ and $\bm v(0)=(2, 2, 2, 2)$, then species $\bm v$ wins the competition; if the initial data  is $\bm u(0)=(5, 5, 5, 5)$, $\bm v(0)=(1, 1, 1, 1)$, then species $\bm u$ wins the competition. This confirms that it is possible to have bi-stability in case (a).  We conjecture that for case (a) we always have $\la^*_1(d_2)<0$ for all $d_2\neq d_1$ and the model has no stable coexistence equilibrium.
\begin{figure}[h]
\centering
\subfiguretopcaptrue
\subfigure[]{
\includegraphics[height=2in]{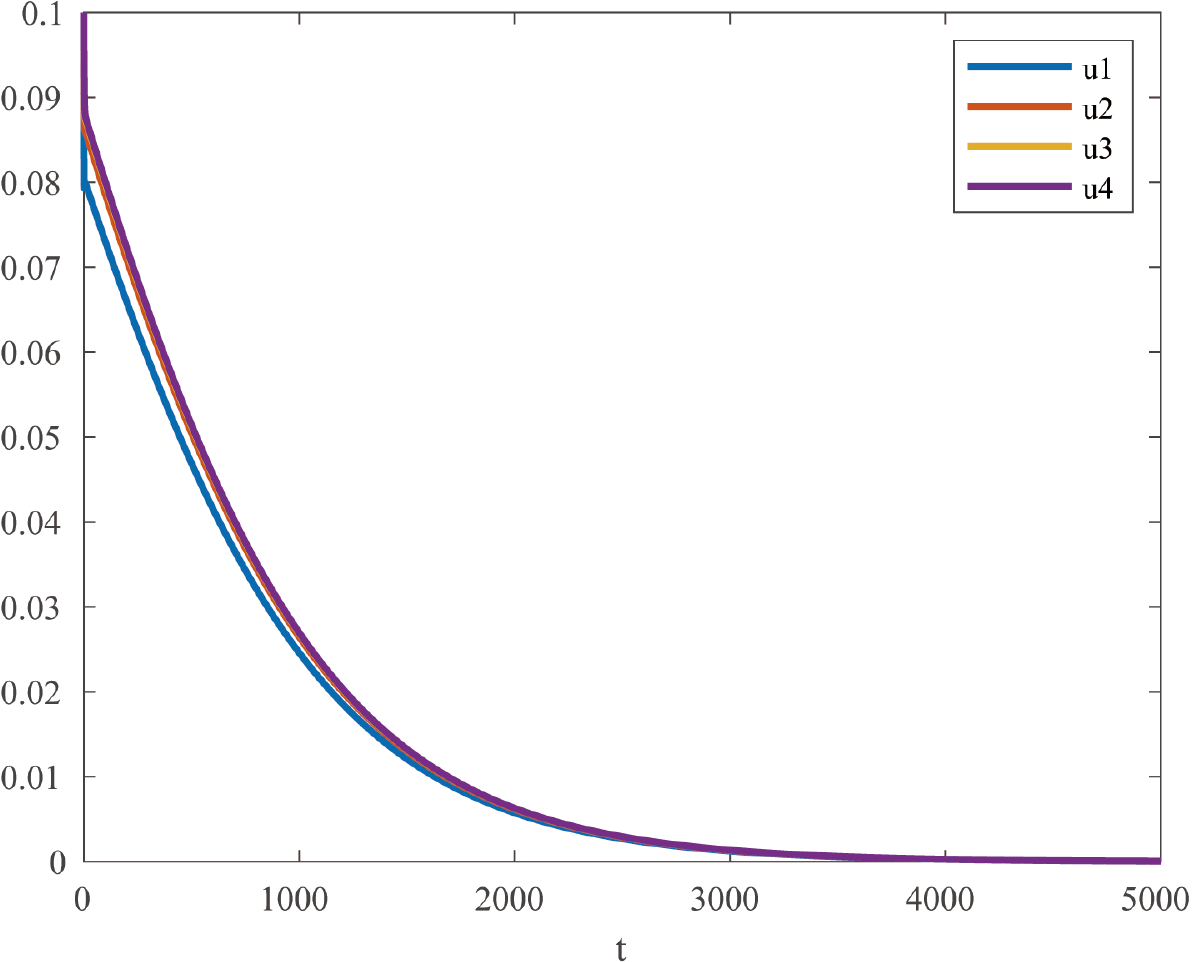}
}
\subfigure[]{
\includegraphics[height=2in]{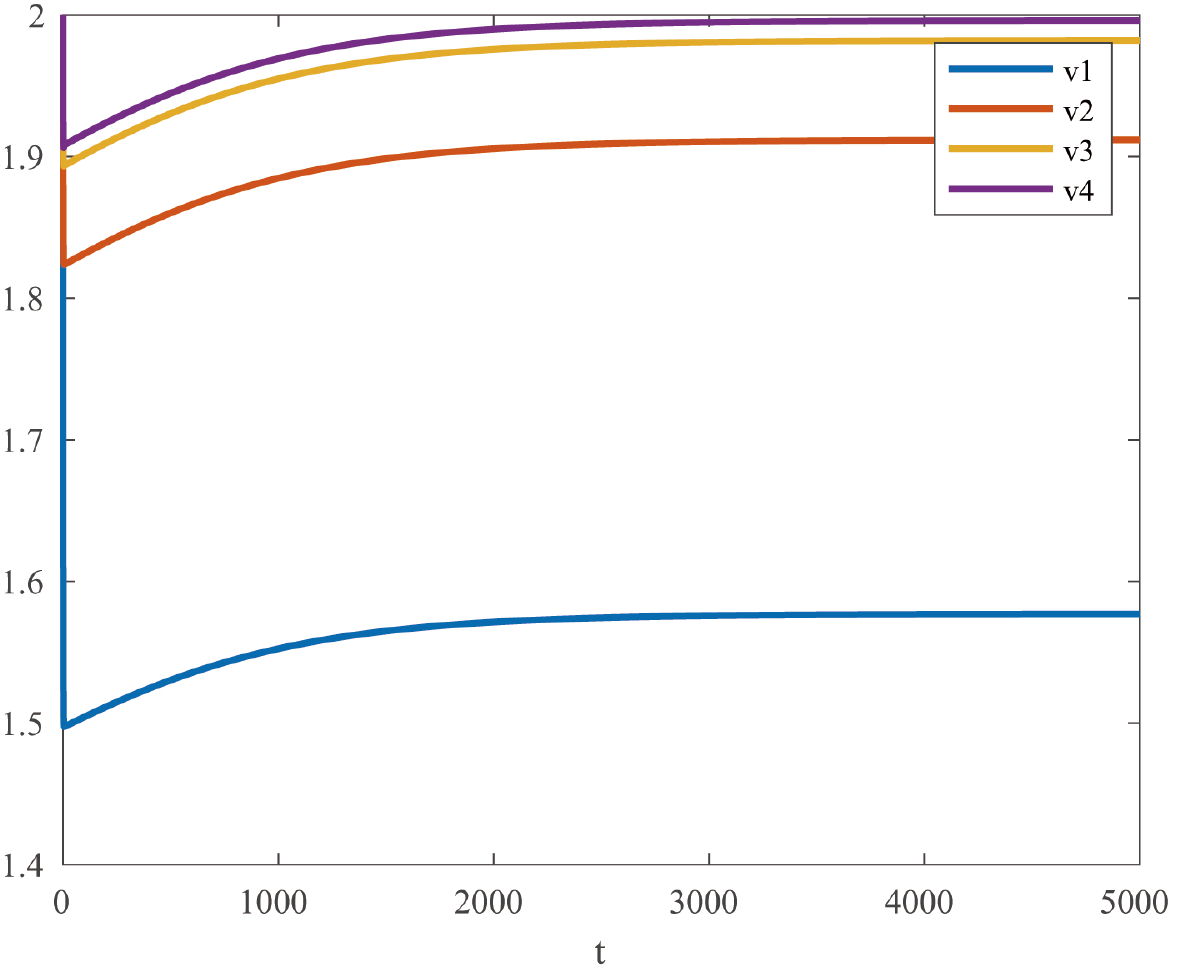}
}
\subfigure[]{
\includegraphics[height=2in]{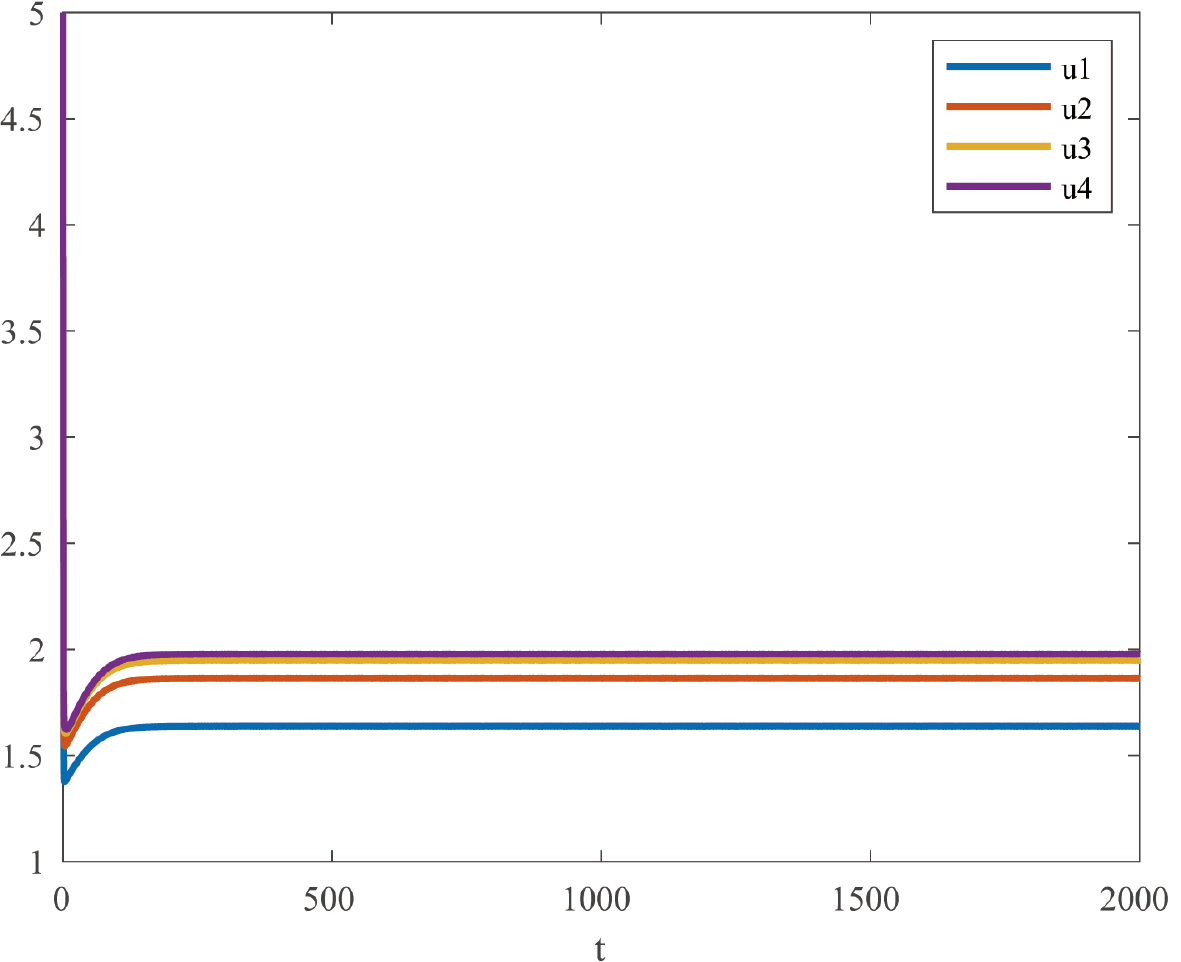}
}
\subfigure[]{
\includegraphics[height=2in]{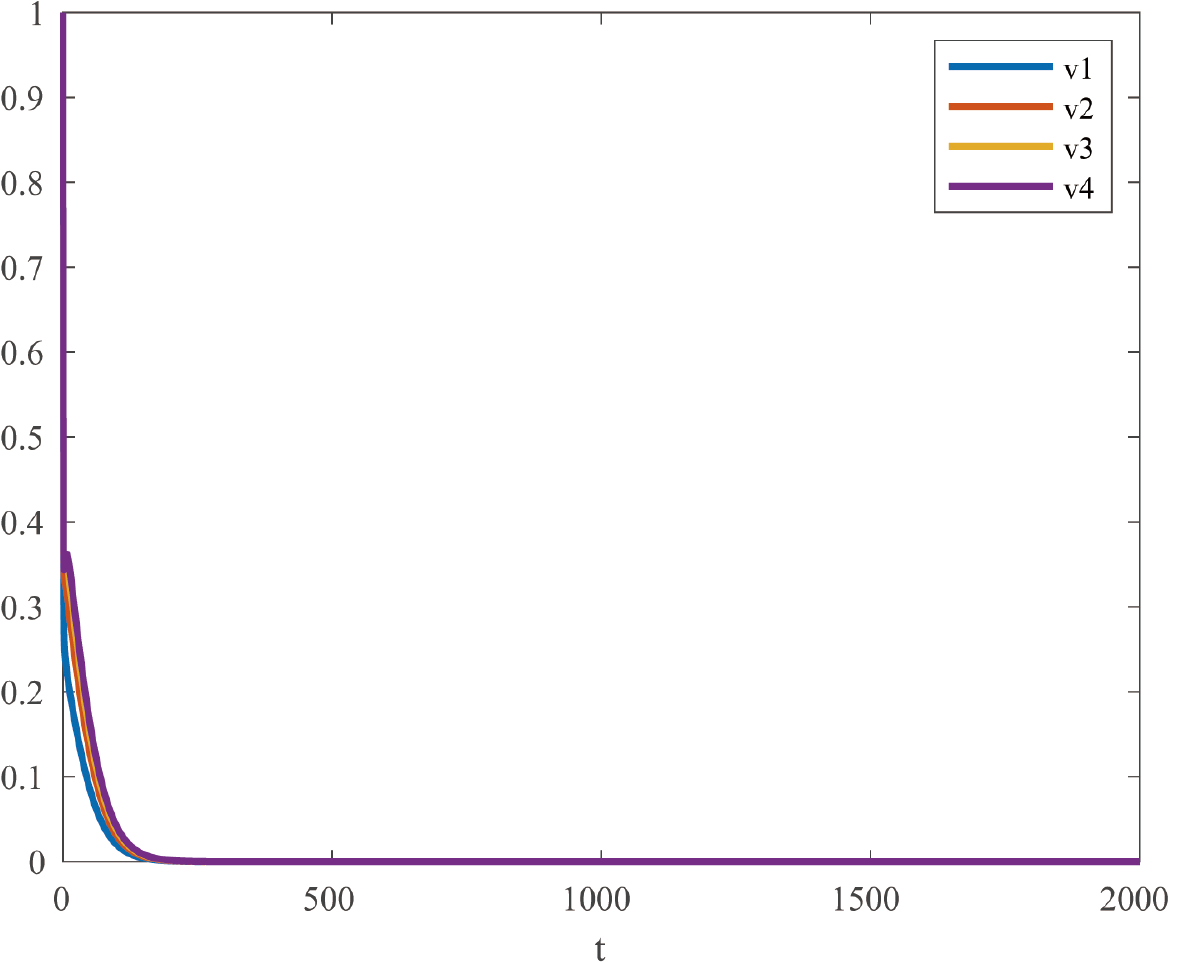}
}
\vskip -15pt
\caption{Solutions of \eqref{pat-cp} with $n=4$ for case (a). The parameters are $r=2$, $d_1=1$, $q_1=0.5$, $d_2=0.08$ and $q_2=0.44$. (a)-(b) Initial data: $\bm u(0)=(0.1, 0.1, 0.1, 0.1)$, $\bm v(0)=(2, 2, 2, 2)$, and species $\bm v$ wins the competition; (c)-(d) Initial data: $\bm u(0)=(5, 5, 5, 5)$, $\bm v(0)=(1, 1, 1, 1)$, and species $\bm u$ wins the competition. This shows that the model has bi-stability in case (a).}
\label{fig_bi}
\vskip -8pt
\end{figure}

  \subsection{Simulations for case (b)}
We choose $d_1=1, q_1=0.5 \text{\ or\ } 3$ such that $E_1$ exists for case (b), and  we  plot the curve $q=q_{\bm{r-u^*}}^*(d)$ in Fig. \ref{figqb0}.
 \begin{figure}[h]
\centering
\subfiguretopcaptrue
\subfigure[]{
\includegraphics[height=2in]{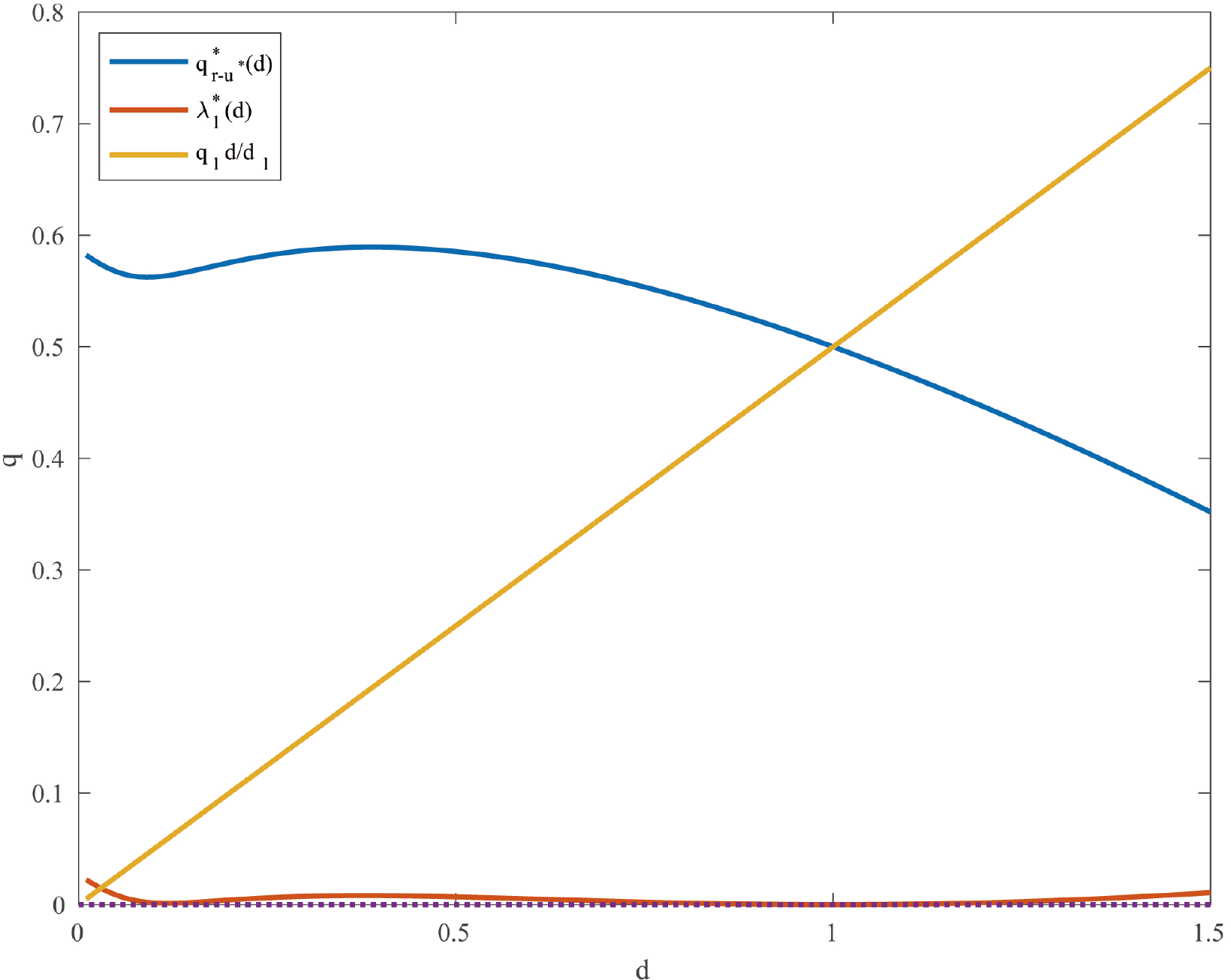}
}
\subfigure[]{
\includegraphics[height=2in]{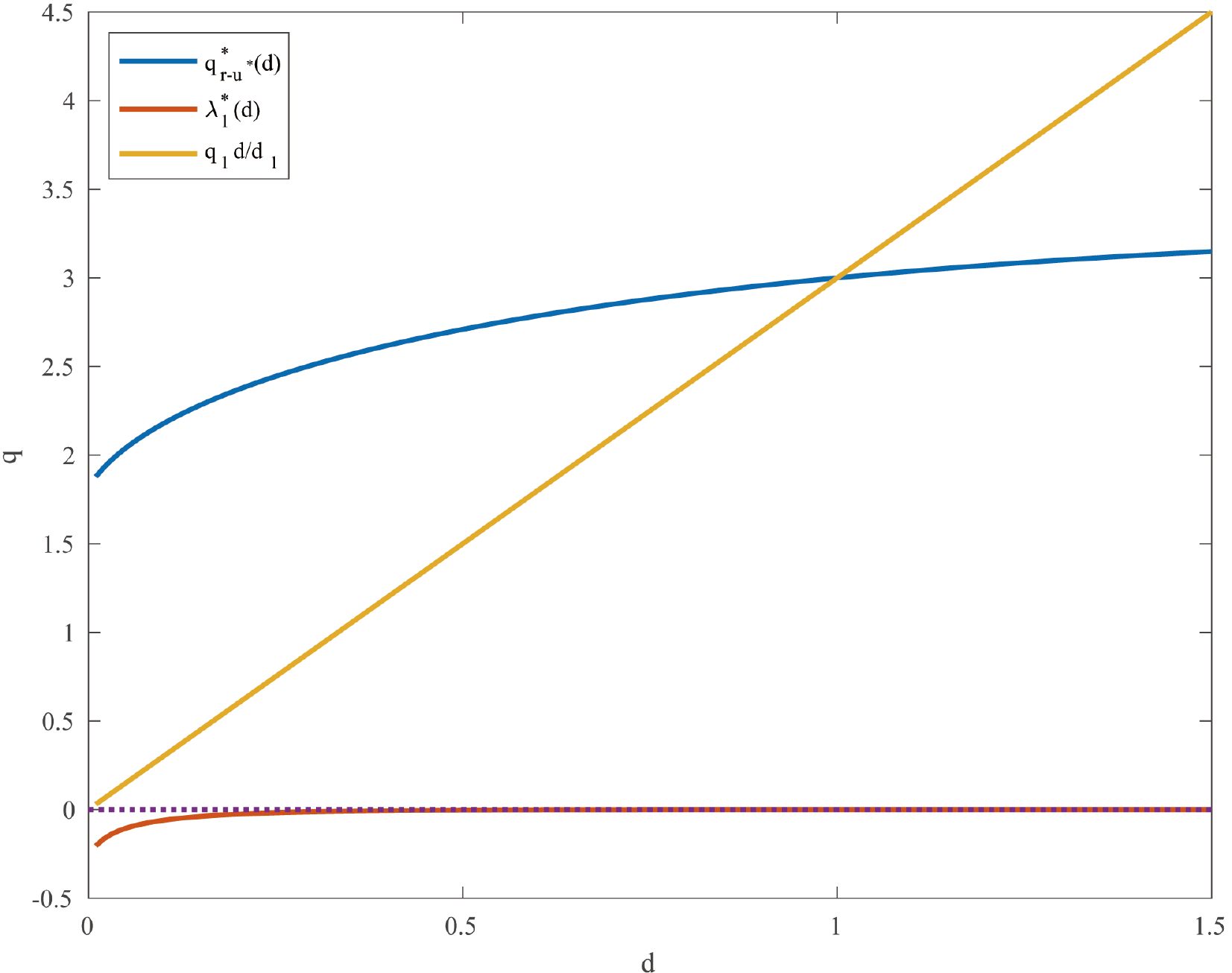}
}
\caption{Curve $q=q_{\bm{r-u^*}}^*(d)$  with $n=4$, $r=2$, $d_1=1$ for case (b). (Left) $q=0.5$; (right) $q=3$. The sign of the curve $\la^*_1(d)$ determines the stability of $E_2$ when $d_2=d$ and $q_2=q_{\bm{r-u^*}}^*(d)$.}
\label{figqb0}
\vskip -8pt
\end{figure}
 \begin{figure}[h]
\centering
\subfiguretopcaptrue
\subfigure[]{
\includegraphics[height=2in]{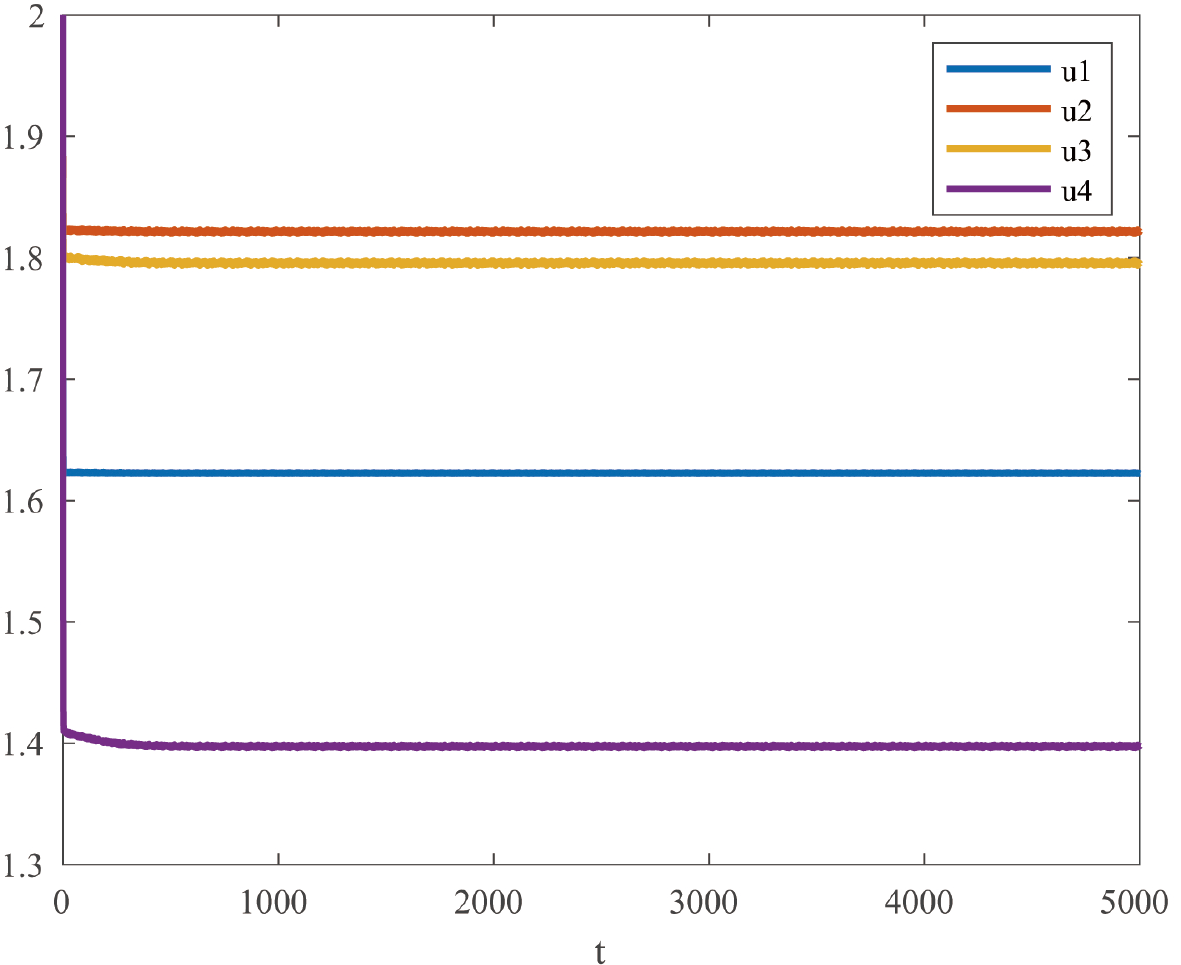}
}
\subfigure[]{
\includegraphics[height=2in]{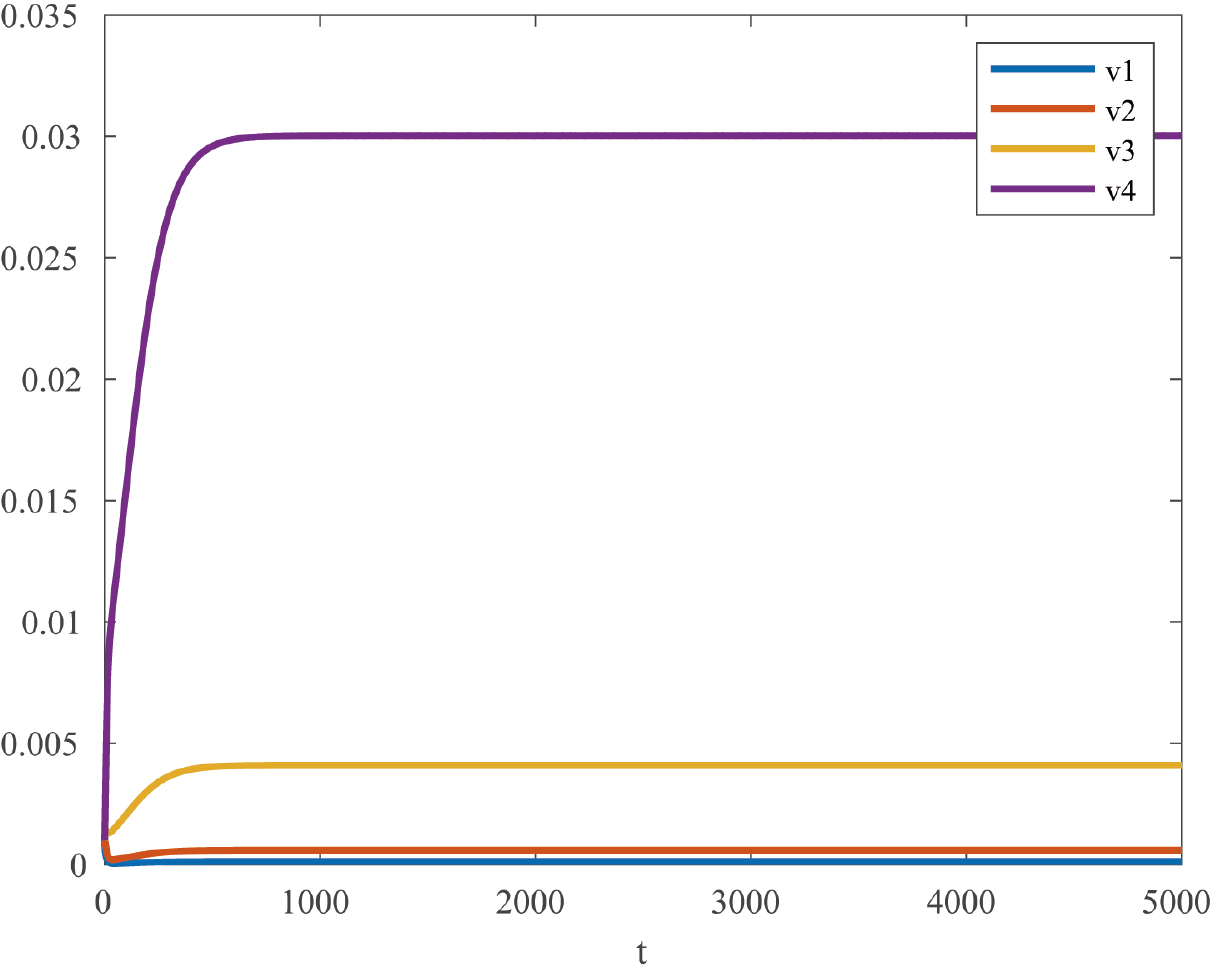}
}
\vskip -15pt
\caption{Solutions of \eqref{pat-cp} with $n=4$ for case (b). The parameters are $r=2$, $d_1=1$, $q_1=0.5$, $d_2=0.05$ and $q_2=0.555$. The two species seem to coexist.}
\label{fig_co}
\vskip -8pt
\end{figure}
By Proposition \ref{theorem_iv2}, $E_1$ is stable if $(d_2, q_2)$ is above the curve  $q=q_{\bm{r-u^*}}^*(d)$ and unstable if it is below the curve. In Fig. \ref{figqb0}, the sign of $\la^*_1(d)$ seems to be positive for $q_1=0.5$ and negative for $q_1=3$.  If $q_1=0.5$ and $(d_2, q_2)$ is sufficiently close  to but below the curve $q=q_{\bm{r-u^*}}^*(d)$, then both
 $E_1$ and  $E_2$ are unstable and we should have coexistence. To confirm this, we choose $(d_2, q_2)=(0.05, 0.555)$. In Fig. \ref{fig_co}, we plot the solutions of \eqref{pat-cp}, and it appears that the two species coexist. If $q_1=3$ and $(d_2, q_2)$ is sufficiently close  to but above the curve $q=q_{\bm{r-u^*}}^*(d)$, we should have bi-stability (we omit the graphs here since they are similar to the ones in case (a) shown in Fig. \ref{fig_bi}). Our simulations show that both bi-stability and coexistence can occur for case (b).


\subsection{Evolution of dispersal for case (a) and (b)}
Suppose that species $\bm v$ is the mutating species, and $(d_2, q_2)$ is close to but not equal to $(d_1, q_1)$.  When the model is coupled with free-flow boundary conditions (case (a)), we always have $\left[q_{\bm{r-u^*}}^*(d)\right]'|_{d=\bar d_1}>0$.
Biologically, this means that the mutating species can invade if and only if it has  a larger diffusion rate.

If  the model is coupled with hostile boundary conditions (case (b)), the dynamics of the model is more complicated. In Fig. \ref{figqb0},  we can see that the sign of $\left[q_{\bm{r-u^*}}^*(d)\right]'|_{d=\bar d_1}$ changes from negative to positive when $q$ increased from $0.5$ to $3$. Biologically,  when the advection rate is small ($q_1=0.5$), the mutating species $\bm v$ can invade if it has a smaller diffusion rate; however when the advection rate is large ($q_1=3$), the mutating species $\bm v$ may need to have a larger diffusion rate than the resident species $\bm u$ to invade it. Therefore if the downstream end is coupled with hostile boundary conditions, whether smaller or larger diffusion rate is a better strategy for the species depends on the advection rate. We conjecture that when $d_1$ is small the sign of $\left[q_{\bm{r-u^*}}^*(d)\right]'|_{d=\bar d_1}$ for case (b)  changes from negative to positive as $q$ increases,  i.e. smaller diffusion rate is better when the advection rate is small, while larger diffusion rate is favored when  advection rate is large.  We also conjecture that when $d_1$ is sufficiently large then the sign of $\left[q_{\bm{r-u^*}}^*(d)\right]'|_{d=\bar d_1}$ for case (b)
is always negative, i.e. smaller diffusion rate is always better.

If we fix $q_1=2$ for case (b), as shown in Fig. \ref{figqb}, the sign of $\left[q_{\bm{r-u^*}}^*(d)\right]'|_{d=\bar d_1}$ changes from positive to negative when $q_1$ increased from $0.5$ to $2$. Therefore, there exists $\bar d_1\in (0.5, 2)$ such that $\left[q_{\bm{r-u^*}}^*(d)\right]'|_{d=\bar d_1}=0$. Moreover, the sign of $\left[q_{\bm{r-u^*}}^*(d)\right]'|_{d=\bar d_1}$ switches from positive to negative at $d_1=\bar d_1$. This suggests that $d_1=\bar d_1$ may be a convergence stable strategy. We conjecture that for each advection rate $q_1\in (0, r)$, there exists a unique intermediate diffusion rate, which is a convergence stable strategy. We remark that if $n=2$ the authors in \cite{xiang2019evolutionarily} have shown that for each $q_1\in (0, r)$ there exists a unique evolutionary stable strategy for $d_1$.

 \begin{figure}[h]
\centering
\subfiguretopcaptrue
\subfigure[]{
\includegraphics[height=2in]{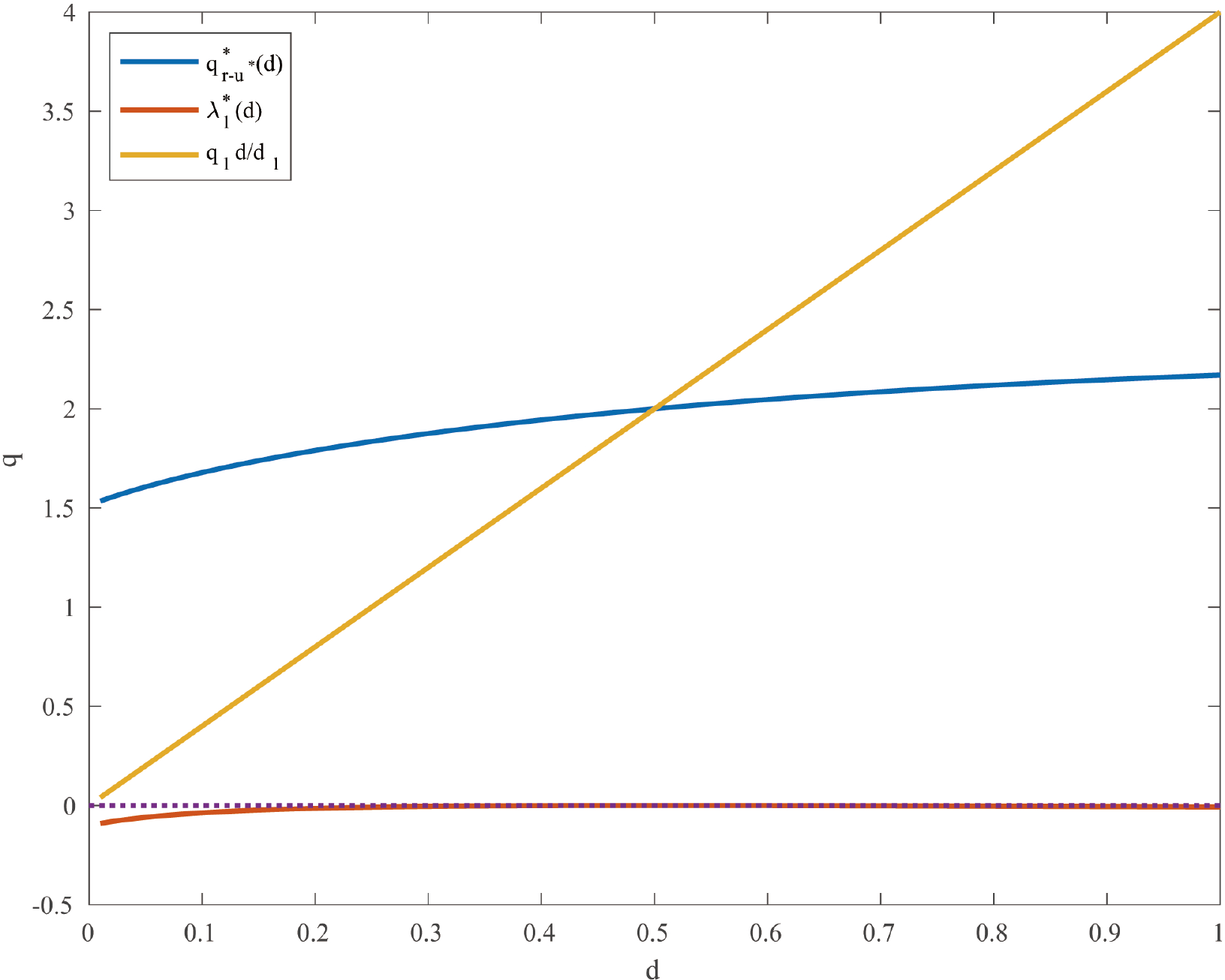}
}
\subfigure[]{
\includegraphics[height=2in]{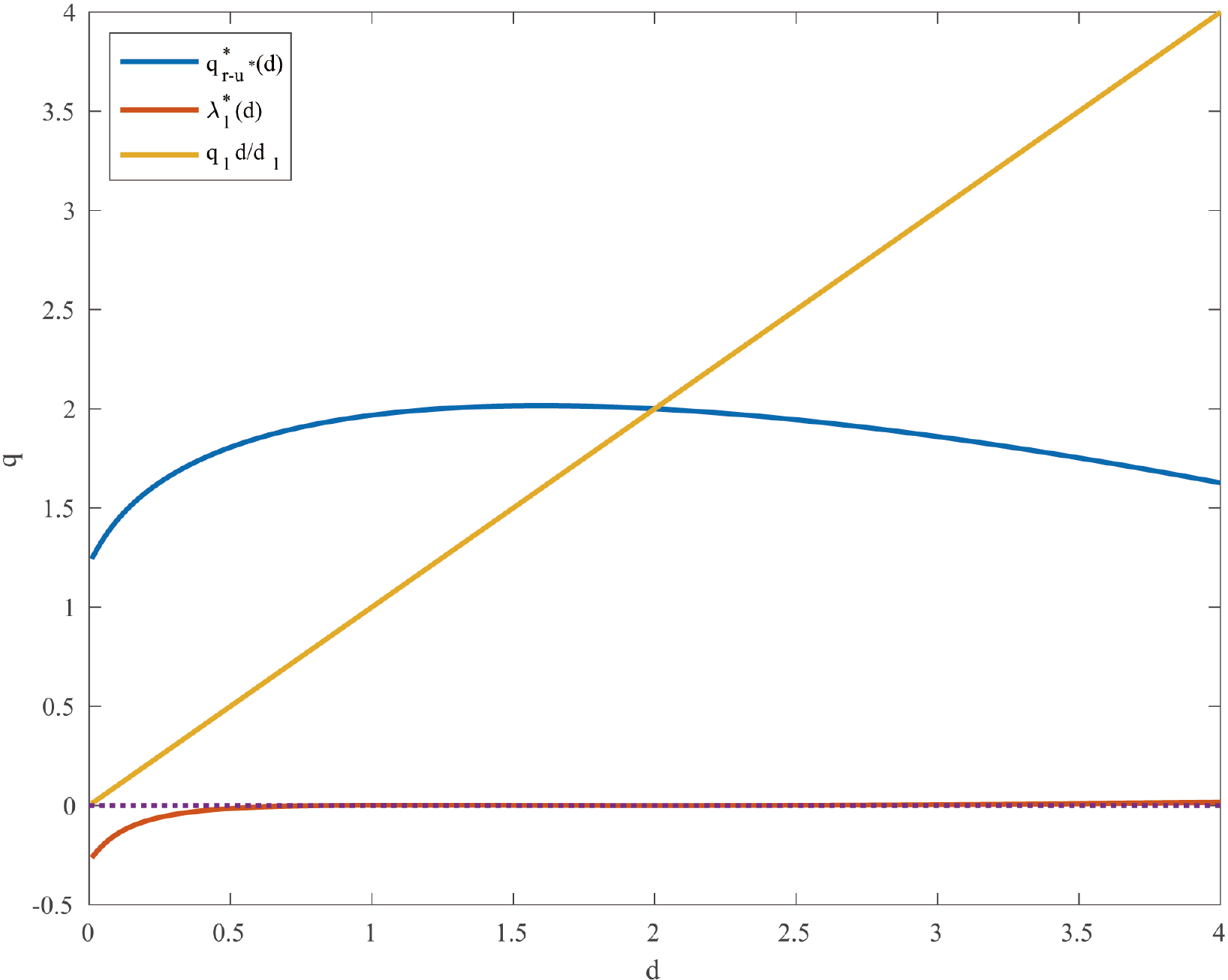}
}
\vskip -15pt
\caption{Curve $q=q_{\bm{r-u^*}}^*(d)$  with $n=4$, $r=2$, $q_1=2$ for case (b). The first figure is for $d_1=0.5$ and the second one for $d_1=2$. The sign of the curve $\la^*_1(d)$ determines the stability of $(\bm 0, \bm v^*)$ when $d_2=d$ and $q_2=q_{\bm{r-u^*}}^*(d)$.}
\label{figqb}
\vskip -8pt
\end{figure}

\FloatBarrier

\vspace{10pt}


\begin{thebibliography}{10}

\bibitem{altenberg2012resolvent}
L.~Altenberg.
\newblock Resolvent positive linear operators exhibit the reduction phenomenon.
\newblock {\em Proc. Natl. Acad. Sci. USA}, 109(10):3705--3710, 2012.

\bibitem{berman1994nonnegative}
A.~Berman and R.~J. Plemmons.
\newblock {\em Nonnegative Matrices in the Mathematical Sciences}, volume~9 of
  {\em Classics in Applied Mathematics}.
\newblock Society for Industrial and Applied Mathematics (SIAM), Philadelphia,
  PA, 1994.

\bibitem{cantrell2004spatial}
R.~S. Cantrell and C.~Cosner.
\newblock {\em Spatial Ecology via Reaction-Diffusion Equations}.
\newblock John Wiley \& Sons, 2004.

\bibitem{stephen2007ideal}
R.~S. Cantrell, C.~Cosner, D.~L. Deangelis, and V.~Padron.
\newblock The ideal free distribution as an evolutionarily stable strategy.
\newblock {\em J. Biol. Dyn.}, 1(3):249--271, 2007.

\bibitem{cantrell2012evolutionary}
R.~S. Cantrell, C.~Cosner, and Y.~Lou.
\newblock Evolutionary stability of ideal free dispersal strategies in patchy
  environments.
\newblock {\em J. Math. Biol.}, 65(5):943--965, 2012.

\bibitem{cantrell2017evolution}
R.~S. Cantrell, C.~Cosner, Y.~Lou, and S.~J. Schreiber.
\newblock Evolution of natal dispersal in spatially heterogeneous environments.
\newblock {\em Math. Biosci.}, 283:136--144, 2017.

\bibitem{chen2021global}
S.~Chen, J.~Shi, Z.~Shuai, and Y.~Wu.
\newblock Global dynamics of a {L}otka-{V}olterra competition patch model.
\newblock {\em Nonlinearity}, 35(2):817--842, 2022.

\bibitem{chen2019spectral}
S.~Chen, J.~Shi, Z.~Shuai, and Y.~Wu.
\newblock Two novel proofs of spectral monotonicity of perturbed essentially
  nonnegative matrices with applications in population dynamics.
\newblock {\em SIAM J. Appl. Math.}, 82(2):654--676, 2022.

\bibitem{cheng2019coexistence}
C.-Y. Cheng, K.-H. Lin, and C.-W. Shih.
\newblock Coexistence and extinction for two competing species in patchy
  environments.
\newblock {\em Math. Biosci. Eng.}, 16(2):909--946, 2019.

\bibitem{cosner1996variability}
C.~Cosner.
\newblock Variability, vagueness and comparison methods for ecological models.
\newblock {\em Bull. Math. Biol.}, 58(2):207--246, 1996.

\bibitem{deangelis2016dispersal}
D.~L. DeAngelis, W.-M. Ni, and B.~Zhang.
\newblock Dispersal and spatial heterogeneity: single species.
\newblock {\em J. Math. Biol.}, 72(1):239--254, 2016.

\bibitem{dieckmann1996dynamical}
U.~Dieckmann and R.~Law.
\newblock The dynamical theory of coevolution: a derivation from stochastic
  ecological processes.
\newblock {\em J. Math. Biol.}, 34(5):579--612, 1996.

\bibitem{dockery1998evolution}
J.~Dockery, V.~Hutson, K.~Mischaikow, and M.~Pernarowski.
\newblock The evolution of slow dispersal rates: a reaction diffusion model.
\newblock {\em J. Math. Biol.}, 37(1):61--83, 1998.

\bibitem{geritz1998evolutionarily}
S.~Geritz, E.~Kisdi, G.~Mesze, and J.~A.~J. Metz.
\newblock Evolutionarily singular strategies and the adaptive growth and
  branching of the evolutionary tree.
\newblock {\em Evol. Biol.}, 12(1):35--57, 1998.

\bibitem{gourley2005two}
S.~A. Gourley and Y.~Kuang.
\newblock Two-species competition with high dispersal: the winning strategy.
\newblock {\em Math. Biosci. Eng.}, 2(2):345--362, 2005.

\bibitem{Hamida2017}
Y.~Hamida.
\newblock The evolution of dispersal for the case of two patches and
  two-species with travel loss.
\newblock Master's thesis, The Ohio State University, 2017.

\bibitem{hastings1983can}
A.~Hastings.
\newblock Can spatial variation alone lead to selection for dispersal?
\newblock {\em Theoret. Population Biol.}, 24(3):244--251, 1983.

\bibitem{hess}
P.~Hess.
\newblock {\em Periodic-Parabolic Boundary Value Problems and Positivity},
  volume 247 of {\em Pitman Research Notes in Mathematics Series}.
\newblock Longman Scientific \& Technical, Harlow, 1991.

\bibitem{hsu1996competitive}
S.~B. Hsu, H.~L. Smith, and P.~Waltman.
\newblock Competitive exclusion and coexistence for competitive systems on
  ordered {B}anach spaces.
\newblock {\em Trans. Amer. Math. Soc.}, 348(10):4083--4094, 1996.

\bibitem{HuangJin}
Q.-H. Huang, Y.~Jin, and M.~A. Lewis.
\newblock {$R_0$} analysis of a {B}enthic-drift model for a stream population.
\newblock {\em SIAM J. Appl. Dyn. Syst.}, 15(1):287--321, 2016.

\bibitem{jiang2020two}
H.~Jiang, K.-Y. Lam, and Y.~Lou.
\newblock Are two-patch models sufficient? {T}he evolution of dispersal and
  topology of river network modules.
\newblock {\em Bull. Math. Biol.}, 82(10):Paper No. 131, 42, 2020.

\bibitem{Jiang-Lam-Lou2021}
H.~Jiang, K.-Y. Lam, and Y.~Lou.
\newblock Three-patch models for the evolution of dispersal in advective
  environments: varying drift and network topology.
\newblock {\em Bull. Math. Biol.}, 83(10):1--46, 2021.

\bibitem{jin2011seasonal}
Y.~Jin and M.~A. Lewis.
\newblock Seasonal influences on population spread and persistence in streams:
  critical domain size.
\newblock {\em SIAM J. Appl. Math.}, 71(4):1241--1262, 2011.

\bibitem{keitt2001allee}
T.~H. Keitt, M.~A. Lewis, and R.~D. Holt.
\newblock Allee effects, invasion pinning, and species' borders.
\newblock {\em The American Naturalist}, 157(2):203--216, 2001.

\bibitem{kirkland2006evolution}
S.~Kirkland, C.-K. Li, and S.~J. Schreiber.
\newblock On the evolution of dispersal in patchy landscapes.
\newblock {\em SIAM J. Appl. Math.}, 66(4):1366--1382, 2006.

\bibitem{lam2015evolution}
K.~Y. Lam, Y.~Lou, and F.~Lutscher.
\newblock Evolution of dispersal in closed advective environments.
\newblock {\em J. Biol. Dyn.}, 9(suppl. 1):188--212, 2015.

\bibitem{lam2016emergence}
K.~Y. Lam, Y.~Lou, and F.~Lutscher.
\newblock The emergence of range limits in advective environments.
\newblock {\em SIAM J. Appl. Math.}, 76(2):641--662, 2016.

\bibitem{lam2016remark}
K.-Y. Lam and D.~Munther.
\newblock A remark on the global dynamics of competitive systems on ordered
  {B}anach spaces.
\newblock {\em Proc. Amer. Math. Soc.}, 144(3):1153--1159, 2016.

\bibitem{LAM2016Munther}
K.-Y. Lam and D.~Munther.
\newblock A remark on the global dynamics of competitive systems on ordered
  {B}anach spaces.
\newblock {\em Proc. Amer. Math. Soc.}, 144(3):1153--1159, 2016.

\bibitem{levin1976population}
S.~A. Levin.
\newblock Population dynamic models in heterogeneous environments.
\newblock {\em Annu. Rev. Ecol. Syst.}, pages 287--310, 1976.

\bibitem{levin1984dispersal}
S.~A. Levin, D.~Cohen, and A.~Hastings.
\newblock Dispersal strategies in patchy environments.
\newblock {\em Theoret. Population Biol.}, 26(2):165--191, 1984.

\bibitem{Li2002JMB}
C.-K. Li and H.~Schneider.
\newblock Applications of {P}erron-{F}robenius theory to population dynamics.
\newblock {\em J. Math. Biol.}, 44(5):450--462, 2002.

\bibitem{li2010global}
M.~Y. Li and Z.~Shuai.
\newblock Global-stability problem for coupled systems of differential
  equations on networks.
\newblock {\em J. Differential Equations}, 248(1):1--20, 2010.

\bibitem{lin2014global}
K.-H. Lin, Y.~Lou, C.-W. Shih, and T.-H. Tsai.
\newblock Global dynamics for two-species competition in patchy environment.
\newblock {\em Math. Biosci. Eng.}, 11(4):947--970, 2014.

\bibitem{lou2019ideal}
Y.~Lou.
\newblock Ideal free distribution in two patches.
\newblock {\em J. Nonlinear Model Anal.}, 2:151--167, 2019.

\bibitem{lou2014evolution}
Y.~Lou and F.~Lutscher.
\newblock Evolution of dispersal in open advective environments.
\newblock {\em J. Math. Biol.}, 69(6-7):1319--1342, 2014.

\bibitem{lou2018coexistence}
Y.~Lou, H.~Nie, and Y.~Wang.
\newblock Coexistence and bistability of a competition model in open advective
  environments.
\newblock {\em Math. Biosci.}, 306:10--19, 2018.

\bibitem{LouNie2018}
Y.~Lou, H.~Nie, and Y.~Wang.
\newblock Coexistence and bistability of a competition model in open advective
  environments.
\newblock {\em Math. Biosci.}, 306:10--19, 2018.

\bibitem{lou2016qualitative}
Y.~Lou, D.-M. Xiao, and P.~Zhou.
\newblock Qualitative analysis for a {L}otka-{V}olterra competition system in
  advective homogeneous environment.
\newblock {\em Discrete Contin. Dyn. Syst.}, 36(2):953--969, 2016.

\bibitem{lou2015evolution}
Y.~Lou and P.~Zhou.
\newblock Evolution of dispersal in advective homogeneous environment: the
  effect of boundary conditions.
\newblock {\em J. Differential Equations}, 259(1):141--171, 2015.

\bibitem{Lu1993}
Z.~Y. Lu and Y.~Takeuchi.
\newblock Global asymptotic behavior in single-species discrete diffusion
  systems.
\newblock {\em J. Math. Biol.}, 32(1):67--77, 1993.

\bibitem{lutscher2006effects}
F.~Lutscher, M.~A. Lewis, and E.~McCauley.
\newblock Effects of heterogeneity on spread and persistence in rivers.
\newblock {\em Bull. Math. Biol.}, 68(8):2129--2160, 2006.

\bibitem{lutscher2007spatial}
F.~Lutscher, E.~McCauley, and M.~A. Lewis.
\newblock Spatial patterns and coexistence mechanisms in systems with
  unidirectional flow.
\newblock {\em Theoret. Population Biol.}, 71(3):267--277, 2007.

\bibitem{lutscher2005effect}
F.~Lutscher, E.~Pachepsky, and M.~A. Lewis.
\newblock The effect of dispersal patterns on stream populations.
\newblock {\em SIAM Rev.}, 47(4):749--772 (electronic), 2005.

\bibitem{ma2020evolution}
L.~Ma and D.~Tang.
\newblock Evolution of dispersal in advective homogeneous environments.
\newblock {\em Discrete Contin. Dyn. Syst.}, 40(10):5815--5830, 2020.

\bibitem{mcpeek1992evolution}
M.~A. McPeek and R.~D. Holt.
\newblock The evolution of dispersal in spatially and temporally varying
  environments.
\newblock {\em The American Naturalist}, 140(6):1010--1027, 1992.

\bibitem{noble2015evolution}
L.~Noble.
\newblock {\em Evolution of Dispersal in Patchy Habitats}.
\newblock PhD thesis, The Ohio State University, 2015.

\bibitem{Owen2001}
M.~R. Owen and M.~A. Lewis.
\newblock How predation can slow, stop or reverse a prey invasion.
\newblock {\em Bull. Math. Biol.}, 63(4):655--684, 2001.

\bibitem{smith2008monotone}
H.~L. Smith.
\newblock {\em Monotone {D}ynamical {S}ystems: {A}n {I}ntroduction to the
  {T}heory of {C}ompetitive and {C}ooperative {S}ystems}.
\newblock American Mathematical Society, Providence, RI, 1995.

\bibitem{speirs2001population}
D.~C. Speirs and W.~S.~C. Gurney.
\newblock Population persistence in rivers and estuaries.
\newblock {\em Ecology}, 82(5):1219--1237, 2001.

\bibitem{vasilyeva2011population}
O.~Vasilyeva and F.~Lutscher.
\newblock Population dynamics in rivers: analysis of steady states.
\newblock {\em Can. Appl. Math. Q.}, 18(4):439--469, 2010.

\bibitem{vasilyeva2012flow}
O.~Vasilyeva and F.~Lutscher.
\newblock How flow speed alters competitive outcome in advective environments.
\newblock {\em Bull. Math. Biol.}, 74(12):2935--2958, 2012.

\bibitem{xiang2019evolutionarily}
J.-J. Xiang and Y.~Fang.
\newblock Evolutionarily stable dispersal strategies in a two-patch advective
  environment.
\newblock {\em Discrete Contin. Dyn. Syst. Ser. B}, 24(4):1875--1887, 2019.

\bibitem{yan2022competition}
X.~Yan, H.~Nie, and P.~Zhou.
\newblock On a competition-diffusion-advection system from river ecology:
  Mathematical analysis and numerical study.
\newblock {\em SIAM J. Appl. Dyn. Syst.}, 21(1):438--469, 2022.

\bibitem{zhao2016lotka}
X.-Q. Zhao and P.~Zhou.
\newblock On a {L}otka-{V}olterra competition model: the effects of advection
  and spatial variation.
\newblock {\em Calc. Var. Partial Differential Equations}, 55(4):Art. 73, 25,
  2016.

\bibitem{zhou2016lotka}
P.~Zhou.
\newblock On a {L}otka-{V}olterra competition system: diffusion vs advection.
\newblock {\em Calc. Var. Partial Differential Equations}, 55(6):Art. 137, 29,
  2016.

\bibitem{zhou2018global}
P.~Zhou and X.-Q. Zhao.
\newblock Global dynamics of a two species competition model in open stream
  environments.
\newblock {\em J. Dyn. Differential Equations}, 30(2):613--636, 2018.

\end{thebibliography}

\end{document}